\documentclass{amsart}
\usepackage{amscd, amssymb, portland, calc, ifthen, amsxtra, pifont}
\usepackage[matrix,arrow,curve,cmtip]{xy}

\usepackage{eucal}
\usepackage{mathrsfs}
\DeclareMathAlphabet{\mathpzc}{OT1}{pzc}{m}{it}

\let \: \colon

\newcommand{\oldmarginpar}[1]{}
\newcommand{\id}{{\mathrm{id}}}

\newcommand{\vbl}{-}
\newcommand{\ang}[1]{{\langle{#1}\rangle}}
\newcommand{\tn}{\otimes}           

\newcommand{\mathbold}{\bf}

\newcommand{\red}[1]{{#1}_{\mathrm{red}}}
\newcommand{\pr}{{\mathrm{pr}}}
\newcommand{\subsec}[1]{{\begin{trivlist}\item\em\large#1\end{trivlist}}}

\newcommand{\longlabelmap}[1]{{\,\buildrel #1\over\longrightarrow\,}}

\newcommand{\longlabelmaps}[2]{{\rightrightarrows}}

\newcommand{\longmap}{{\,\longrightarrow\,}}
\newcommand{\an}{\mathrm{an}}

\DeclareMathOperator{\End}{End}
\DeclareMathOperator{\Hom}{Hom}

\DeclareMathOperator{\Aut}{Aut}
\DeclareMathOperator{\sHom}{\underline{Hom}}
\DeclareMathOperator{\sAut}{\underline{Aut}}
\DeclareMathOperator{\tr}{tr}

\DeclareMathOperator{\GL}{GL}
\DeclareMathOperator{\SL}{SL}

\DeclareMathOperator{\Spec}{Spec}

\newcommand{\Gal}{{\mathrm{Gal}}}
\newcommand{\sO}{{\mathcal{O}}}
\newcommand{\gr}{{\mathrm{gr}}}

\newcommand{\bF}{{\mathbold F}}
\newcommand{\bA}{{\mathbold A}}
\newcommand{\bP}{{\mathbold P}}
\newcommand{\bQ}{{\mathbold Q}}

\newcommand{\bZ}{{\mathbold Z}}
\newcommand{\bC}{{\mathbold C}}

\newcommand{\bN}{{\mathbold N}}
\newcommand{\Gm}{{\mathbold G}_{\mathrm{m}}}
\newcommand{\Ga}{{\mathbold G}_{\mathrm{a}}}
\newcommand{\et}{{\text{{\rm \'et}}}}

\newcommand{\m}{{\mathfrak{m}}}

\newcommand{\Fr}{{\mathrm{Fr}}}

\renewcommand{\geq}{\geqslant}

\DeclareMathOperator{\colim}{\mathrm{colim}}

\newcommand{\smcoprod}{{\,\scriptstyle\amalg\,}}
\newcommand{\smsmcoprod}{{\,\scriptscriptstyle\amalg\,}}
\newcommand{\ab}{\mathrm{ab}}

\newcommand{\rightlabelxyarrows}[2]{{\ar@<0.7ex>^-{#1}[r]\ar@<-0.7ex>_-{#2}[r]}}

\newcommand{\LRlabelxyarrows}[2]{{\ar@<0.7ex>^-{#1}@{<-}[r]\ar@<-0.7ex>_-{#2}[r]}}

\newcommand{\displaylabelfork}[6]{{	\entrymodifiers={+!!<0pt,\fontdimen22\textfont2>}
	\def\objectstyle{\displaystyle}
\xymatrix{{#1} \ar^-{#2}[r] & {#3} \ar@<0.7ex>^-{#4}[r]\ar@<-0.7ex>_-{#5}[r] & {#6}}}}

\newcommand{\displaylabelcofork}[6]{{	\entrymodifiers={+!!<0pt,\fontdimen22\textfont2>}
	\def\objectstyle{\displaystyle}
\xymatrix{{#1} \ar@<0.7ex>^-{#2}[r]\ar@<-0.7ex>_-{#3}[r] & {#4} \ar^-{#5}[r] & {#6}}}}

\newcommand{\displaylabelrightarrows}[4]{{\entrymodifiers={+!!<0pt,\fontdimen22\textfont2>}
	\def\objectstyle{\displaystyle}
\xymatrix{{#1} \ar@<0.7ex>^-{#2}[r]\ar@<-0.7ex>_-{#3}[r] & {#4}}}}

\newcommand{\marpar}[1]{}
\newcommand{\marpartd}[1]{}
\newcommand{\hide}[1]{}
\newcommand{\comment}[1]{}
\newcommand{\tempcomment}[1]{}

\newcommand{\coeff}{s}

\newcommand{\setof}[2]{\{#1 \;|\; #2\}}

\newcommand{\bcp}{\odot}
\newcommand{\ptst}{E}

\newcommand{\gh}[1]{\gamma_{#1}}

\newcommand{\aff}{\mathsf{Aff}}

\newcommand{\Space}{\mathsf{Sp}}
\newcommand{\Set}{\mathsf{Set}}
\newcommand{\CommMon}{\mathsf{CommMon}}

\newcommand{\nset}{{\bN^{(\ptst)}}}

\newcommand{\qbar}{\bar{\bQ}}
\newcommand{\Het}{H_{\mathrm{\et}}}
\newcommand{\Hetc}{H_{\mathrm{\et,c}}}
\newcommand{\Hc}{H_{\mathrm{c}}}
\newcommand{\Frob}{{\mathrm{Frob}}}
\newcommand{\fp}{\mathfrak{p}}

\newcommand{\wnus}[1]{{W_{#1}^{*}}}
\newcommand{\wnls}[1]{W_{{#1*}}}
\newcommand{\wius}{{\wnus{\infty}}}
\newcommand{\wils}{{\wnls{\infty}}}
\newcommand{\wus}{{\wnus{}}}
\newcommand{\wls}{{\wnls{}}}

\newcommand{\ltm}{{v}}
\newcommand{\fftm}{{f}}
\newcommand{\ffaus}{{f^*}}
\newcommand{\ffalsh}{{f_!}}
\newcommand{\stmap}{{s}}

\newcommand{\shfcond}{T}

\newcommand{\Ring}{\mathsf{Ring}}

\newcommand{\inv}[1]{{{#1}^{\mathrm{inv}}}}

\newcommand{\ncomp}{N^{\mathrm{c}}}
\newcommand{\zcomp}{Z^{\mathrm{c}}}
\newcommand{\Lambtil}{\tilde{\Lambda}}

\newtheoremstyle{mythm}{}{}%
  {\itshape}
  {}
  {\bfseries}
  {}
  { }
  {\thmnumber{#2.\hspace{1.5mm}}\thmname{#1}\thmnote{ #3}.}

\newtheoremstyle{intro}{}{}%
  {\itshape}
  {}
  {\bfseries}
  {}
  { }
  {\thmname{#1}\thmnumber{ #2}\thmnote{ #3}.}

\newtheoremstyle{myrmk}{}{}%
  {}
  {}
  {}
  {}
  { }
  {{\bfseries\thmnumber{#2.\hspace{1.5mm}}}{\itshape\thmname{#1}}\thmnote{ #3}.}

\numberwithin{equation}{subsection}

\theoremstyle{mythm}
\newtheorem{thm}[subsection]{Theorem}
\newtheorem{theorem}[subsection]{Theorem}

\newtheorem{proposition}[subsection]{Proposition}
\newtheorem{lemma}[subsection]{Lemma}

\newtheorem{corollary}[subsection]{Corollary}

\theoremstyle{myrmk}

\theoremstyle{intro}

\theoremstyle{plain}
\newtheorem*{prop*}{Proposition}
\newtheorem*{cor*}{Corollary}
\newtheorem*{conj*}{Conjecture}

\theoremstyle{definition}

\makeatletter
\def\@seccntformat#1{\@ifundefined{#1@cntformat}%
{\csname the#1\endcsname\quad}
{\csname #1@cntformat\endcsname}
}
\def\section@cntformat{\thesection.\enspace}
\def\subsection@cntformat{\thesubsection.}
\makeatother

\newcommand\mnote[1]{}

\newcounter{hour}\newcounter{minute}
\newcommand{\printtime}{\setcounter{hour}{\time/60}%
        \setcounter{minute}{\time-\value{hour}*60}%
        \ifthenelse{\value{hour}<10}{0}{}\thehour:%
        \ifthenelse{\value{minute}<10}{0}{}\theminute}

\begin{document}

\title[$\Lambda$-rings and the field with one element]
	{$\Lambda$-rings and the field with one element}
\author[J.~Borger]{James Borger}
\address{Australian National University, Canberra}

\begin{abstract}
	The theory of $\Lambda$-rings, in the sense of Grothendieck's Riemann--Roch theory,
	is an enrichment of the theory of commutative
	rings.  In the same way, we can enrich usual algebraic
	geometry over the ring $\bZ$ of integers to produce
	$\Lambda$-algebraic geometry.  We show that $\Lambda$-algebraic geometry is
	in a precise sense an algebraic geometry over a
	deeper base than $\bZ$ and that it has many properties predicted for algebraic
	geometry over the mythical field with one element.  Moreover, it does this
	is a way that is both formally robust and closely related to active areas in
	arithmetic algebraic geometry.
\end{abstract}
\date{\today. \printtime.}
\thanks{This is a preliminary version, made available in advance of the Frobenius
Lifts workshop at the Lorentz Center, Leiden, 2009 October 5--9. I hope to include
the toric uniformization theorem in the final version.}
\thanks{This work was partly supported by Discovery Project DP0773301, 
a grant from the Australian Research Council.}
\email{borger@maths.anu.edu.au}

\maketitle

\section*{Introduction}

Many writers have mused about algebraic geometry over deeper bases than the ring $\bZ$ of
integers. Although there are several, possibly unrelated reasons for this, here I will mention
just two. The first is that the combinatorial nature of enumeration formulas in linear algebra
over finite fields $\bF_q$ as $q$ tends to $1$ suggests that, just as one can work over all
finite fields simultaneously by using algebraic geometry over $\bZ$, perhaps one could bring in
the combinatorics of finite sets by working over an even deeper base, one which somehow allows
$q=1$. It is common, following Tits~\cite{Tits:F1}, to call this mythical base $\bF_1$, the
field with one element. (See also Steinberg~\cite{Steinberg:F1}, p.\ 279.) The second purpose is
to prove the Riemann hypothesis. With the analogy between integers and polynomials in mind, we
might hope that $\Spec \bZ$ would be a kind of curve over $\Spec \bF_1$, that
$\Spec\bZ\tn_{\bF_1}\bZ$ would not only make sense but be a surface bearing some kind of
intersection theory, and that we could then mimic over $\bZ$ Weil's
proof~\cite{Weil:Riemann-hyp-function-field} of the Riemann hypothesis over function
fields.\footnote{The origins of this idea are unknown to me. Manin~\cite{Manin:Lectures-on-zeta}
mentions it explicitly. According to Smirnov~\cite{Smirnov:letter-to-Manin}, the idea occurred
to him in 1985 and he mentioned it explicitly in a talk in Shafarevich's seminar in 1990. It may
well be that a number of people have had the idea independently since the appearance of Weil's
proof.} Of course, since $\bZ$ is the initial object in the category of rings, any theory of
algebraic geometry over a deeper base would have to leave the usual world of rings and
schemes.

The most obvious way of doing this is to consider weaker algebraic structures than rings
(commutative, as always), such as commutative monoids, and to try using them as the affine building
blocks for a more rigid theory of algebraic geometry. This has been pursued in a number of papers,
which I will cite below. Another natural approach is motived by the following question, first
articulated by Soul\'e~\cite{Soule:F1}: Which rings over $\bZ$ can be defined over $\bF_1$? Less
set-theoretically, on a ring over $\bZ$, what should descent data to $\bF_1$ be?

The main goal of this paper is to show that a reasonable answer to this question is a
\emph{$\Lambda$-ring structure}, in the sense of Grothendieck's Riemann--Roch
theory~\cite{Grothendieck:Chern}. More precisely, we show that a $\Lambda$-ring structure on a
ring can be thought of as descent data to a deeper base in the precise sense that it gives rise
to a map from the big \'etale topos of $\Spec\bZ$ to a $\Lambda$-equivariant version of the big
\'etale topos of $\Spec\bZ$, and that this deeper base has many properties expected of the field
with one element. Not only does the resulting algebraic geometry fit into the supple formalism
of topos theory, it is also arithmetically rich---unlike the category of sets, say, which is the
deepest topos of all. For instance, it is closely related to global class field theory, 
complex multiplication, and crystalline cohomology.

So let us define an $\bF_1$-algebra to be a $\Lambda$-ring. (The language of $\bF_1$ will quickly
feel silly, but for most of this paper, it will be useful as an expository device.) More
generally, define an $\bF_1$-scheme to be a scheme equipped with a $\Lambda$-structure. The
theory of $\Lambda$-structures on schemes was introduced in
Borger~\cite{Borger:BGWV}\cite{Borger:SLAG}. (Although see
Grothendieck~\cite{Grothendieck:Pursuing-stacks}, p.\ 506.) Defining a $\Lambda$-structure on a
general scheme $X$ takes some time, as it does on a general ring. But when $X$ is flat over
$\bZ$, there is a simple equivalent definition: it is a commuting family of endomorphisms
$\psi_p\:X\to X$, indexed by the set of prime numbers $p$, such that each $\psi_p$ agrees with
the $p$-th power Frobenius map on the special fiber $X\times_{\Spec\bZ}\Spec\bF_p$. It is also
true that any reduced $\Lambda$-scheme is flat over $\bZ$. Keeping these two facts in mind, it
is possible to read most of this paper without knowing the definition in general.

If we take this $\bF_1$-to-$\Lambda$ dictionary seriously, then the functor that removes the
$\Lambda$-structure from a $\Lambda$-scheme should be thought of as removing the descent data,
and hence as the base-change functor from $\bF_1$ to $\bZ$.
So for instance, $\bF_1$ itself should be defined to be $\bZ$ with its unique
$\Lambda$-structure, where each $\psi_p$ is the identity map. It is the initial
object in the category of $\Lambda$-rings.

As indicated above, the definition of a $\Lambda$-structure extends not just to the category of
schemes but to the entire ambient topos. By this, I mean the big \'etale topos over $\bZ$, by
definition the category of sheaves of sets on the category of affine schemes equipped with the
\'etale topology. It is therefore natural to define the big \'etale topos over $\bF_1$ to be the
category of such sheaves with $\Lambda$-structure. The base-change functor $\ltm^*$ that strips
off the $\Lambda$-structure then induces a well-defined map $$ \ltm\:\Spec\bZ\longmap\Spec\bF_1,
$$ not on the level of schemes, which would of course be meaningless, but on the level of big
\'etale toposes.

The functor $\ltm^*$ has not only a right adjoint $\ltm_*$, as required, but also a left adjoint
$\ltm_!$. If we think of $\ltm^*$ as the base-change functor, we should think of $\ltm_*$ as the
Weil restriction of scalars functor and $\ltm_!$ as the base-forgetting functor. In terms of
definitions rather than interpretations, $\ltm_*$ sends a space to its arithmetic jet space, which
is a multi-prime, algebraic version of Buium's $p$-jet space~\cite{Buium:Arithmetic-diff-equ},
which is in turn a formal $p$-adic lift of the Greenberg
transform~\cite{Greenberg:I}\cite{Greenberg:II}. On the other hand, $\ltm_!$ sends a space to its
space of big Witt vectors. Both of these are given their natural $\Lambda$-structures. These
constructions are exotic by the standards of algebraic geometry over fields, but even in the
arithmetically local case, where we consider $p$-adic schemes with only one Frobenius lift
$\psi_p$, they have proven very useful. See, for example, Buium~\cite{Buium:Arithmetic-diff-equ}
for applications of the first and Illusie~\cite{Illusie:dRW-1016} for the second. (In fact, Buium
has had similar ideas regarding Frobenius lifts and the field with one element.
See~\cite{Buium:geometry-of-fermat-adeles}, the preface of~\cite{Buium:Arithmetic-diff-equ}, or
Buium--Simanca~\cite{Buium-Simanca:Arithmetic-Laplacians}. Manin~\cite{Manin:Cyclotomy} has also
recently interpreted Witt vectors as being related to the field with one element.)

For applications to usual, non-$\Lambda$ arithmetic algebraic geometry, the most interesting
spaces over $\bF_1$ are those obtained from schemes over $\bZ$. Whether we produce them by
applying $\ltm_!$ or $\ltm_*$, the result is almost never a scheme of finite type over $\bF_1$.
So from the perspective of algebraic geometry over $\bZ$, the spaces of principal interest over
$\bF_1$ are not those of finite type. In fact, as has been expected, there appear to be very few
schemes of finite type over $\bZ$ that descend to $\bF_1$ at all. But their
infrequency is so extreme that they become interesting on their own terms. For example,
it might be possible to describe the category of algebraic spaces of finite type over $\bF_1$ in
purely combinatorial terms, without any mention of algebraic geometry or $\Lambda$-structures.

This question is probably within reach, and the second purpose of this paper is to take some
steps in that direction. For example, we prove the following theorem, stated here under
restricted hypotheses:

\begin{thm}\label{thm:intro-A}
	Let $X$ be a smooth proper scheme over $\bZ[1/M]$, for some integer $M\geq 1$.
	If $X$ descends to $\bF_1$ (i.e., admits a $\Lambda$-structure), 
	then is has the following properties:
	\begin{enumerate}
		\item The action of $\Gal(\bar\bQ/\bQ)$ on 
			any $p$-adic \'etale cohomology group 
			$\Het^n(X_{\bar\bQ},\bQ_p)$ factors through its abelianization.
		\item There is an integer $N$, all of whose prime divisors divide $Mp$,
			such that the restriction of the representation
			$\Het^n(X_{\bar\bQ},\bQ_p)$ to $\Gal(\bar\bQ/\bQ(\zeta_N))$
			is isomorphic to a sum of powers of the cyclotomic character.
		\item The $(i,j)$ Hodge number of the mixed Hodge
			structure on the singular cohomology $H^n(X^{\an},\bC)$	is zero when $i\neq j$.
	\end{enumerate}
\end{thm}

Another way to express this theorem is that only abelian Artin--Tate motives can be defined over
$\bF_1$. (Compare Soul\'e~\cite{Soule:F1}, 6.4, question 4.) While this means there are no
motivically interesting $\bF_1$-schemes of finite type, the argument that shows this is rather
interesting. In the zero-dimensional case, it is an elementary consequence of the Kronecker--Weber
theorem and Chebotarev's density theorem. (See Borger--de
Smit~\cite{Borger-deSmit:integral-lambda-models}.) In the higher-dimensional case, the argument
also uses the Lefschetz theorem and the proper base change theorem for \'etale cohomology, on the
one hand, and on the other, $p$-adic Hodge theory, including the potentially semi-stable theorem
for arbitrary varieties proved by Kisin~\cite{Kisin:pst}, following the work of many others.

The theorem above attests to the predicted combinatorial nature of schemes of finite type over
$\bF_1$. I emphasize once again that the combinatorial nature is not built into the foundations
of the theory---it is a consequence of hard arithmetic results in the presence of finiteness
conditions. Indeed, I expect that the cohomological theory of infinite-dimensional spaces over
$\bF_1$ contains, via $\ltm_!$ and de~Rham--Witt theory, the full theory of motives, and even in
a visible way. It is hard to imagine a theory defined in combinatorial terms having this
property.

As strong as the cohomological restrictions above are, they are far from being sharp. In a
future version of this paper, I hope to show that all examples of finite type come from toric
varieties, in a certain precise sense. Here is a partial result in this direction whose proof
does appear in the present version. It is probably even a necessary step in the proof of the
full classification theorem.

\begin{thm}\label{thm:intro-B}
	Let $X$ be a $\Lambda$-scheme which is of finite type over $\bZ$. 
	Then $X$ has a point with coordinates in a cyclotomic field. 
	When $X$ is proper, there is even a $\Lambda$-morphism $\Spec\bZ\to X$.
\end{thm}

In other words, every $\bF_1$-scheme of finite type has a cyclotomic point and, if proper, has
an $\bF_1$-point. I believe that this result had not been predicted. Observe that in
the affine case, the theorem is a new result about $\Lambda$-rings, which
can be stated independently of the theory. It says that every
$\Lambda$-ring which is finitely generated as a ring admits a ring map to some cyclotomic field.
Even in this case, the proof uses all the deep arithmetic results above.

If we apply theorem~\ref{thm:intro-A} and the Lefschetz fixed-point formula at various primes
$p$, we obtain the following result, which was essentially predicted (implicitly
by the followers of Tits~\cite{Tits:F1}, and by definition in 
Kurokawa~\cite{Kurokawa:zeta-functions-over-F1}): 

\begin{corollary}\label{cor:intro-A1}
	Let $X$ be a smooth proper scheme over $\bZ[1/M]$, for some integer $M\geq 1$.
	If $X$ descends to $\bF_1$, then 
	there is an integer $N\geq 1$ divisible only by the primes
	dividing $M$ such that for all prime powers $q>1$ with $q\equiv 1\bmod N$,
	the number of $\bF_q$-valued points of $X$ is $P(q)$, where $P$ is the
	Hodge polynomial of $X$: $P(t)=\sum_{i,n}(-1)^n r_{i,n}t^i$, where
	$r_{i,n}$ is the $(i,i)$ Hodge number of $\Hc^n(X^{\an},\bC)$.
	
	In particular, if $X$ is smooth and proper over $\bZ$, the number of such
	points is given by the polynomial $P(q)$ for all prime powers $q>1$.
\end{corollary}

Let us now consider $\bF_1$-points. It is not difficult to show, when $X$ is separated and of
finite type over $\bF_1$, that $X(\bF_1)$ is finite. So it is tempting to hope that the number
of $\bF_1$-valued points would be $P(1)$, the Euler characteristic of $X$. This is typically
false, but if we restrict to complemented points---those whose complement is an open
sub-$\bF_1$-scheme of $X$---then it is true in many situations. For instance, it is true for
toric varieties, when given a certain natural, ``toric'' $\Lambda$-structure. But it is not
always true. It fails for the toric varieties $\bA^1$ and $\bP^1$ if they are given the
Chebychev $\Lambda$-structure (\ref{subsec:F1-points-on-Chebychev}). On the other hand, if all
$\Lambda$-varieties can indeed be built from toric varieties with the toric $\Lambda$-structure,
it might be possible to salvage the formula $X(\bF_1)=P(1)=\chi(X)$ by cleverly reinterpreting
the definitions.

Our final purpose is to consider variations on the definition of $\Lambda$-structure, such as
function-field analogues. For example, instead of working over $\Spec\bZ$, we can work over a
smooth curve $S$ over a finite field $k$. Then our family of commuting Frobenius lifts would be
indexed by the closed points of $S$, and the same procedure as over $\bZ$ gives the notion of a
$\Lambda_S$-structure, a topos of $\bF_1^{S}$-spaces, and a topos map 
$\ltm_{S}\:S\longmap\Spec{\bF_1^S}$. 
Having invented an absolute algebraic geometry relative to $S$, we can ask how it relates to the
usual absolute algebraic geometry relative to $S$, that is, algebraic geometry over $k$. The
answer is that the structure map $\stmap\:S\to\Spec k$ factors naturally
as a composition of topos maps:
	$$
	\xymatrix{
	S\ar^-{\ltm_S}[rr]\ar^-{\stmap}[dr]
		&	& \Spec{\bF^S_1}\ar@{-->}_-{\fftm}[dl] \\
		& \Spec k.
	}
	$$
The map $\fftm$ is not an equivalence, but it is not far from being one. For example, $\fftm^*$
embeds the category of reduced algebraic spaces over $k$ fully faithfully into the category of
spaces over $\bF_1^S$, but there can be schemes defined over $\bF_1^S$ (given by certain
rank-one Drinfeld modules) that do not descend to $k$. As a check to see if this could lead to a
proof of the Riemann hypothesis for $\bZ$, one might examine the translation of Weil's proof of
the Riemann hypothesis for $S$ from algebraic geometry over $k$ to that over $\bF_1^S$.
But since
the current version of our theory says nothing about the archimedean place of $\bQ$, it is
hard to imagine this succeeding without further ideas. Even so, it should be done.

Let us now list the contents of this paper. In section~\ref{sec:lambda-spaces}, we recall the
foundations of the theory of spaces over $\bF_1$ (that is, $\Lambda$-spaces). In
section~\ref{sec:examples}, we give examples of $\bF_1$-schemes. In
section~\ref{sec:sub-lambda-spaces}, we discuss sub-$\bF_1$-spaces and in particular $\bF_1$-valued
points. In section~\ref{sec:function-spaces}, we discuss function spaces over $\bF_1$ and
especially $\GL_n$. In section~\ref{sec:abelian-motives}, we show that abelian motives are
Artin--Tate. This is a result in usual, non-$\Lambda$ number theory needed in the following
section. In section~\ref{sec:p-adic-cohomology}, we discuss the $p$-adic \'etale cohomology of
$\bF_1$-schemes of finite type and implications for point counting. And in
section~\ref{sec:variations}, we consider variations on our approach to $\bF_1$ for function
fields and number fields larger than $\bQ$.

\section*{Other work}

In the early days of this project, I was greatly inspired by Manin's
exposition~\cite{Manin:Lectures-on-zeta}, Soul\'e's paper~\cite{Soule:F1}, and Deninger's
program, for example ~\cite{Deninger:cohomological-approach}. In a strict mathematical sense,
this paper does not owe them much, but their spiritual effect has been deep.

There are, of course, many approaches to absolute algebraic geometry which I did not mention
above. Here are some I know about. One is that an $\bF_1$-algebra should be some kind of
algebraic structure that is set-theoretically weaker than a commutative ring, for example a
commutative monoid. From this point a view an $\bF_1$-vector space is often taken to be a set,
perhaps with some additional weak structure. One could investigate the $K$-theory that comes out
of this, and even aspire to see the place at infinity by incorporating archimedean information
in these structures. Another approach has been to pursue notions of $\zeta$-functions over
$\bF_1$, perhaps independently of any formal definition of $\bF_1$. A final approach is to find
and prove analogues over number fields of basic geometric results over function fields.

For these approaches see the following references:
Baez~\cite{Baez:TWF184},
Connes--Consani~\cite{Connes-Consani:notion-of-geometry-over-f1},
Connes--Consani--Marcoli~\cite{Connes-Consani-Marcoli:fun-with-f1},
Deitmar~\cite{Deitmar:F1-schemes}\cite{Deitmar:F1-and-toric}\cite{Deitmer:zeta-and-k-theory}, 
Diers~\cite{Diers:book}, 
Durov~\cite{Durov:new-approach},
Haran~\cite{Haran:non-additive-geometry} \cite{Haran:mysteries-of-the-real}, 
Kapranov~\cite{Kapranov:absolute-direct-image},
Kapranov--Smirnov~\cite{Kapranov-Smirnov:F1}, 
Kurokawa~\cite{Kurokawa-Koyama:multiple-zeta}\cite{Kurokawa:mutlple-zeta-an-example}\cite{Kurokawa:zeta-functions-over-F1},
Kurokawa--Ochiai--Wakayama~\cite{Kurokawa-et-al:absolute-derivations},
Smirnov~\cite{Smirnov:Hurwitz},
Tate--Voloch~\cite{Tate-Voloch:linear-forms},
and To\"en--Vaqui\'e~\cite{Toen-Vaquie:Under-Spec-Z}.
Several of these writers have other papers on the subject, but I believe these
are representative of their approaches.

\section*{Acknowledgments}

My focused work on $\Lambda$-algebraic geometry began on October 22, 2003, when I read an email
from Ivan Fesenko asking if there were relations between my paper~\cite{Borger-Wieland:PA} with
Ben Wieland and the field with one element. That question instantly gave direction to
some scattered thoughts about $\Lambda$-algebraic geometry and the $p$-adic absolute
point (\ref{subsec:local-absolute-point}).  My interests at the time were 
arithmetically local, and as obvious as the global connection is in retrospect,
it had not occurred to me before that day.
So I thank him greatly.

I would also like to thank Mark Kisin for many conversations about this project over several
years. Most of his influence on this project is on forthcoming work, but even with this paper,
if he had not insisted that I begin writing up what I knew in the case of
$\Lambda$-varieties of finite type over $\bZ$, I never would have uncovered much of what is
here. And since the foundational papers~\cite{Borger:BGWV}\cite{Borger:SLAG} owe their existence
to the present paper, perhaps they also owe some of it to him.

I did some of this work in 2004--2005 at the Institut des Hautes \'Etudes Scientifiques
and the Max-Planck-Institut f\"ur Mathematik. It is a pleasure to thank them
for their generous support and nearly ideal working environments.  I did more
work on this topic in January 2006, as a visitor at the University of Chicago.  I thank 
Alexander Beilinson and Vladimir Drinfeld for making that possible.

I have given a number of lectures on this material, and this paper has benefited from
questions many people have asked, especially Alexandru Buium, Jordan Ellenberg,
and Kirsten Wickelgren.  I would particularly like to thank Wickelgren
for some questions that directly inspired the material in
section~\ref{sec:function-spaces}.


\tableofcontents



\section{$\Lambda$-spaces}
\label{sec:lambda-spaces}

\subsection{} {\em General $\Lambda$-spaces.}
Let us quickly review~\cite{Borger:SLAG}. Let $\wus=\wius$ 
denote the infinite-length big Witt vector functor
from the category of spaces to itself. (The category of spaces is, by definition,
the category of sheaves of sets on the category of affine schemes under the \'etale
topology.) Then $\wus$ carries a monad structure, and a $\Lambda$-structure on a space $X$ 
is by definition an action of $\wus$ on $X$.  Note that if $X$ is an algebraic space,
then the space $\wus(X)$ is ind-algebraic but typically not algebraic. 
This is just the familiar fact that the ring of infinite-length Witt vectors is naturally
a projective limit of rings.

Here are some important examples. If $A$ is a ring, then we have $\wus(\Spec A)=
\colim_n\Spec W_n(A)$, where $W_n$ denotes the functor of (big) Witt vectors
of length~$n$.  Therefore a $\Lambda$-structure on the space $\Spec A$ is the same as a
$\Lambda$-ring structure on $A$ in the usual sense. If $X$ is a flat algebraic space over
$\bZ$, then $\Lambda$-structure on $X$ is the same as a commuting family of endomorphisms
$\psi_p\:X\to X$, one for each prime $p$, such that $\psi_p$ agrees with the $p$-th power
Frobenius map on $X\times_{\Spec\bZ}\Spec\bF_p$.  
If a reduced algebraic space admits a $\Lambda$-structure,
then it must be flat over $\bZ$.
Finally, if $X$ is a $\Lambda$-algebraic space, then $\red{X}$ is a closed
$\Lambda$-algebraic subspace.
(Because a $\Lambda$-structure is given by a monad action, a subspace of a $\Lambda$-space
can have at most one compatible $\Lambda$-structure.)
Non-reduced $\Lambda$-algebraic spaces appear only sporadically in this paper.

The functor $\wus$ has a right adjoint $\wls=\wils$, the arithmetic
jet space functor,
which is a generalized version of Buium's
$p$-jet space functor~\cite{Buium:Arithmetic-diff-equ}. 
If $A$ is a ring, then $\wls(\Spec A) = \Spec \Lambda\bcp A$, where
$\Lambda\bcp A$ denotes the $\Lambda$-ring freely generated by $A$. For example,
$\wls(\bA^1_{\bZ})=\Spec\Lambda$, where $\Lambda$ is the free $\Lambda$-ring on one
generator, the ring of symmetric functions in infinitely many variables.

Let $S=\Spec\bZ$, and let $\Space_S=\Space_{\bZ}$ 
denote the category of spaces. Let $\Space_{S/\Lambda}$
denote the category of $\Lambda$-spaces (with $\Lambda$-equivariant, or rather
$\wus$-equivariant, morphisms). This can also be described using $\wls$. By 
adjunction, $\wls$ inherits a comonad structure from the monad structure on $\wus$. The
category of $\Lambda$-spaces is the same as the category of spaces equipped with
an action of the comonad $\wls$. It is therefore a topos, like $\Space_S$.

\subsection{} \emph{Categorical structure.}
Let $\ltm^*\:\Space_{S/\Lambda}\to\Space_S$ denote the functor that simply strips off the
$\Lambda$-structure. The point of this paper is that $\ltm^*$ can also be thought of as
the functor that strips off the descent data from $\bZ$ to $\bF_1$, and hence that it can be
thought of as the base-change functor from $\bF_1$ to $\bZ$.  Therefore,
we have the following equation:
\begin{equation*}
	\xymatrix@R=20pt@C=40pt{
	\Space_{S}\ar@/^2pc/[dd]^{\ltm_*}\ar@/_2pc/[dd]_{\ltm_!} & &
	\Space_{\bZ}\ar@/^2pc/[dd]^{\txt{\scriptsize Weil\\ \scriptsize restrict}}
		\ar@/_2pc/[dd]_{\txt{\scriptsize forget\\ \scriptsize base}} \\ 
	 & = \quad& \\
	\Space_{S/\Lambda}\ar[uu]_
	{\ltm^*}
	& & \Space_{\bF_1}.\ar[uu]|{S\times_{\bF_1}\vbl}
	}		
\end{equation*}
This means that the structure on the right is defined to be that on the left, or that
we think of the precisely defined left-hand side using the geometric language of the
right-hand side.  Each functor is the left adjoint of the one to its right.
The left adjoint $\ltm_!$ of $\ltm^*$ sends $X$ to its Witt space
$\wus(X)$ with the natural $\Lambda$-structure, and the right adjoint $\ltm_*$ sends $X$
to its arithmetic jet space
$\wls(X)$, again with the natural $\Lambda$-structure. As always the left adjoint of a
base-change functor is called base-forgetting, and the right adjoint is called 
Weil restriction of scalars.  In particular, if one accepts the premise of this paper,
the space 
$\Spec\bZ\times_{\Spec\bF_1}\Spec\bZ$ must be defined to be the Witt space $\wus(\Spec\bZ)$.

These three adjoint functors form, by definition, an essential topos map
$\ltm\:\Space_\bZ\to\Space_{\bF_1}$.  This is what one would would hope to have with any
algebraic geometry over a deeper base than $\Spec \bZ$. Similarly, the base-forgetting
functor $\ltm_!$ is faithful but not full. (See~\cite{Borger:BGWV}, 12.2.)

Yet another way of expressing the point of this paper, in the playful tradition of
the field with one element, is the nonsense formula
	``$\Lambda = \bZ[\Gal(\bZ/\bF_1)]$''.
The meaning of this is that if descent from $\bZ$ to $\bF_1$ were controlled by a finite
group, one would call it $\Gal(\bZ/\bF_1)$.  In that case, descent for rings would
alternatively be controlled by the plethory $\bZ[\Gal(\bZ/\bF_1)]$ whose
underlying ring would be the polynomial ring freely generated by the set $\Gal(\bZ/\bF_1)$.
(See Borger--Wieland~\cite{Borger-Wieland:PA}.)  But the plethory that actually
controls this is $\Lambda$.  So while the group $\Gal(\bZ/\bF_1)$ does
not exist, the polynomial algebra it would generate if it did exist does.

Note that everything above extends to the case where $S$ is the spectrum of the ring of integers
of any number field or any smooth curve over a finite field. Thus for any such $S$, there is a
topos $\Space_{\bF_1^S}=\Space_{S/\Lambda_S}$, the topos of spaces over the $S$-variant of the
field with one element. These toposes are all related as $S$ varies, and as one would expect,
$S=\Spec\bZ$, as above, gives the deepest one. We will return to this in
section~\ref{sec:variations}.

\subsection{} \emph{$\Lambda$-modules.}
An important topic we will not discuss in this paper is the $\Lambda$-analogue of a
quasi-coherent sheaf. The non-linear nature of $\Lambda$-structures makes module theory slightly
subtler than it is in equivariant algebraic geometry under actions of monoids or Lie algebras,
or more generally in the context of To\"en--Vaqui\'e~\cite{Toen-Vaquie:Under-Spec-Z}. The reason
for this, in the language of Borger--Wieland~\cite{Borger-Wieland:PA}, is that the additive
bialgebra of the plethory $\Lambda$ does not agree with the cotangent algebra of $\Lambda$.
Therefore one cannot properly speak about modules without specifying certain extra information.
In the case of $\Lambda$, these are essentially the slopes of the Frobenius operators. I mention
this here only because $\bF_1$-modules are a frequent concern in papers on the field with one
element and I will not address them.

\section{Examples}
\label{sec:examples}

\subsection{} {\em The point.}
The ring $\bZ$ has a unique $\Lambda$-structure---each $\psi_p$ is the identity.
Under this structure, it is the initial object in the category of $\Lambda$-rings. It
is therefore reasonable to denote it $\bF_1$ and call $\Spec \bZ$, viewed as a
$\Lambda$-space, the absolute point.

\subsection{} {\em Monoid algebras.} 
\label{subsec:monoid-algebras}
If $M$ is any 
commutative monoid, the 
monoid algebra $\bZ[M]$ has a natural $\Lambda$-action induced by 
$\psi_p\colon m\mapsto m^p$ for any $m\in M$, and so $\Spec\bZ[M]$ descends naturally to 
$\bF_1$.  I will call this $\Lambda$-action the {\em toric} $\Lambda$-action.  
For example, given a choice of coordinates, 
$\bA^r\times\Gm^s$ equals $\Spec \bZ[\bN^r\times \bZ^s]$, which descends naturally to $\bF_1$.

In fact, we have even more: the monoid scheme
structure on $\Spec\bZ[M]$ also descends to $\bF_1$, 
as does the group structure if $M$ is a group.  Indeed, the coalgebra structure on $\bZ[M]$ 
given by $m\mapsto m\tn m$ for all $m\in M$ is a map of $\Lambda$-rings.  The same is true 
of the counit and, if $M$ is a group, the antipode $m\mapsto m^{-1}$.   For example, all 
split tori can be thought of as group schemes over $\bF_1$.  The group scheme 
$\mu_n=\Spec\bZ[x]/(x^n-1) = \Spec\bZ[\bZ/n\bZ]$ 
also descends to $\bF_1$. Following Kapranov--Smirnov~\cite{Kapranov-Smirnov:F1}, Soul\'e
calls this the base change to $\bZ$ of $\bF_{1^n}$, the field with $1^n$ 
elements~\cite{Soule:F1}.

\subsection{} {\em Limits and colimits.}
\label{subsec:limits}
The category of $\Lambda$-rings has products and 
coproducts, and their underlying rings agree with the same constructions taken in the 
category of rings.  
In fact, this is true for all limits and colimits---in particular 
pull-backs, push-outs, direct limits, and inverse limits.
This is because the forgetful functor $\Ring_{\Lambda}\to\Ring_{\bZ}$
has a left and a right adjoint.

Since it is a topos, $\Space_{S/\Lambda}$ also has all limits and colimits. And since
$\ltm^*$ has a left and a right adjoint, the space underlying any limit, or colimit, of
$\Lambda$-spaces agrees with the limit, or colimit, or the underlying spaces. Further,
since algebraization is a left adjoint, it commutes with $\wus$. 
\marpar{check} 
Therefore the algebraization of a $\Lambda$-space is an algebraic $\Lambda$-space. In
particular, given a system of algebraic spaces with $\Lambda$-actions, the colimit taken
in the category of algebraic spaces has a unique compatible $\Lambda$-action. (The
analogous fact for limits is true simply because the subcategory of $\Space_S$ consisting
of algebraic spaces is closed under limits.)

For the same reason,
the affinization of an algebraic $\Lambda$-space is an affine $\Lambda$-space.
\marpar{check}

\subsection{} {\em Toric varieties.} A toric variety is a colimit of spectra of 
monoid rings, where the maps in the system are open immersions
induced by maps of the monoids.  Therefore the colimit in $\Space_S$ is
an algebraic space (and even a scheme).  By~\ref{subsec:monoid-algebras}
and~\ref{subsec:limits},
they carry natural $\Lambda$-actions.  Thus toric varieties descend to $\bF_1$. 
(Compare Soul\'e~\cite{Soule:F1}.) 
In particular, projective spaces $\bP^n$ do, once we choose homogeneous coordinates.
Then the Frobenius lifts are given by 
	$$
	\psi_p\colon [x_0,\dots,x_n]\mapsto [x_0^p,\dots,x_n^p].
	$$

\subsection{} {\em The Chebychev line.}  
\label{subsec:chebychev-line}
Clauwens~\cite{Clauwens:Line} has used Ritt's
work~\cite{Ritt:Permutable}\cite{Ritt:Permutable-errata} to argue that, up to isomorphism, the
affine line $\Spec\bZ[x]$ has exactly one $\Lambda$-structure besides the toric one
of~\ref{subsec:monoid-algebras}. It can be described as follows. The ring $\bZ[t^{\pm 1}]$,
endowed with the toric $\Lambda$-action, has a $\Lambda$-involution $t\mapsto t^{-1}$.
By~\ref{subsec:limits}, the fixed subring is naturally a $\Lambda$-ring. It is freely generated
as a ring by $x=t+t^{-1}$, and this gives the other $\Lambda$-action on $\bZ[x]$. It also has a
simple $K$-theoretic interpretation as the representation ring of the algebraic group $\SL_2$,
but in this paper we are regarding the connection between $\Lambda$-rings and $K$-theory as a
curiosity. The polynomials $\psi_p(x)$ are Chebychev polynomials:
	\[
	\psi_2(x) = x^2-2, \quad \psi_3(x) = x^3-3x, \quad \psi_5(x)=x^5-5x^3+5x, \quad \dots.
	\]

More generally, any subgroup of $\GL_n(\bZ)$ acts on the toric $\Lambda$-ring 
$\bZ[x_1^{\pm 1},\dots,x_n^{\pm 1}]$, and the invariant subrings give more examples of 
$\Lambda$-rings.  For example, if we take the permutation representation 
$S_n\to\GL_n(\bZ)$ of the $n$-th symmetric group, the invariant subring is
isomorphic to $\bZ[\lambda_1,\dots,\lambda_{n-1},\lambda_n^{\pm 1}]$ and thus
gives a non-toric $\Lambda$-action on $\bA^{n-1}\times\Gm$.

It would be interesting to generalize Clauwens' results.  For example, are there
\marpar{delete?}
only finitely many isomorphism classes of $\Lambda$-structures on $\bA^2$?

\subsection{} {\em Singular lines.} 
\label{subsec:singular-lines}
We can divide out, not just by group actions, as in~\ref{subsec:chebychev-line},
but also by any $\Lambda$-equivalence relation, whether we take the quotient
in the category of affine 
$\Lambda$-spaces, algebraic $\Lambda$-spaces, or all $\Lambda$-spaces.  
For instance, on $\Gm$ with the toric $\Lambda$-action, we can identify $1$ and $-1$ 
to make a nodal line 
	\[
	A = \setof{f(z)\in\bZ[t^{\pm 1}]}{f(-1)=f(1)},
	\]
or we can identify $1$ with itself to
order two to make a cuspidal line
	\[
	A' = \setof{f(z)\in\bZ[t^{\pm 1}]}{f'(1)=0}.
	\]
It is easy to check that these subrings of $\bZ[t^{\pm 1}]$ are sub-$\Lambda$-rings.

Observe that we can identify the points $1,q\in\Gm$ only when $q=\pm 1$---otherwise, the
$\psi_p$ operators would fail to descend to the quotient. But if we use $\Lambda$-algebraic
parameter spaces, there are non-discrete moduli. For instance, in the family
	$$
	\Gm\times\Gm \longlabelmap{\pr_1} \Gm,
	$$
we can identify the diagonal section $\Delta$ and the identity section $\id_{\Gm}\times 1$, to
get the family of nodal lines $\Gm/(1\sim q)$ parameterized by $q\in\Gm$. The result is a
perfectly legitimate $\Lambda$-algebraic family of nodal $\Lambda$-quotients of $\Gm$. The
reason why before $q$ could only lie in a finite set is that there we insisted that the
endomorphisms $\psi_p$ act trivially on the parameter space. We will see below that an algebraic
$\Lambda$-space of finite type over $\bZ$ has only finitely many points (with coordinates in
$\bC$, say) fixed by the $\psi_p$ operators~(\ref{pro:finiteness-of-periodic-spaces}). Therefore
in any $\Lambda$-algebraic family, only finitely many of the fibers will be stable under the
$\psi_p$. In particular this will be true for any universal family, assuming the
$\Lambda$-moduli space is of finite type. So the finiteness phenomenon above is rather
general.

We can also contract any $\Lambda$-invariant modulus on $\Gm$.  If $f(x)$ is a product
of polynomials the form $x^n-1$, then the two maps
	$$
	\bZ[x^{\pm 1}] \rightrightarrows \bZ[x^{\pm 1}]/(f(x)),
	$$
one given by $x\mapsto x$ and the other by $x\mapsto 1$, are $\Lambda$-ring maps.
Therefore their equalizer is a $\Lambda$-ring.  Its spectrum is $\Gm$ with the 
zero locus of $f(x)$ contracted to a point.

Last, these constructions can be used to make $\Lambda$-schemes that cannot be covered by open
affine $\Lambda$-schemes. In particular, it is inaccurate to say that $\Lambda$-schemes are formed
by gluing $\Lambda$-rings together in the Zariski topology, though it is generally the right idea.
The following example is due to Ben Wieland. Consider $\bP^1$ with the toric $\Lambda$-structure,
and let $X$ be the quotient by the $\Lambda$-equivalence relation $0\sim \infty$. Then there is no
open immersion $j\:U\to X$ which has the following properties: $U$ is an affine $\Lambda$-scheme,
$j$ is a $\Lambda$-map, and the nodal point $0=\infty$ is in the image. Indeed, if there were such
a neighborhood $U$, then since it would be affine, the set $U(\bC)$ would be the complement in
$X(\bC)$ of a finite nonempty subset $T$ of $\Gm(\bC)$. But $U\cap\Gm$ would also be a
sub-$\Lambda$-space of $\Gm$. Since $\psi_n$ on $\Gm$ is the $n$-th power map, $T$ would have to be
closed under the extraction of $n$-th roots, which is not the case for any finite nonempty subset
of $\bC^*$.

\subsection{} {\em Zero-dimensional varieties.}  As shown in
Borger--de~Smit~\cite{Borger-deSmit:integral-lambda-models}, it follows from
class field theory that every 
$\Lambda$-ring which is both finite over $\bZ$ and reduced is contained in a product of 
cyclotomic fields.  It is in fact isomorphic to a sub-$\Lambda$-ring of a product of 
toric $\Lambda$-rings of the form $\bZ[x]/(x^n-1)$.  
This can be viewed as an integral version of the Kronecker--Weber theorem.
Thus even zero-dimensional $\Lambda$-algebraic geometry is somewhat interesting.
In fact, the proofs of several theorems about higher-dimensional $\Lambda$-varieties use
the zero-dimensional theory in key ways. \marpar{More precise?}

\subsection{} {\em Non-example: flag varieties.} 
It follows from a theorem of Paranjape and Srinivas~\cite{Panjarape-Srinivas:self-maps} that no
flag varieties besides projective spaces $\bP^n$ admit even one Frobenius lift $\psi_p$.
Therefore, besides $\bP^n$, flag varieties are not defined over $\bF_1$.

In~\ref{cor:lambda-varieties-are-Artin-Tate} below, we show that there are strong motivic
conditions on a variety for it to descend to $\bF_1$. But in the case of flag varieties, the
obstruction is not in the motive. Indeed, flag varieties are paved by affine spaces and are
therefore indistinguishable from them from the point of view of motives. It would be interesting to
know whether flag varieties admit a weakened version of a $\Lambda$-structure but which
is still stronger than being paved by $\Lambda$-varieties.
\marpar{Grodal in appendix?}

\subsection{} {\em Non-example: curves of genus $g\geq 1$.}
Let $C_{\bQ}$ be a connected smooth proper curve over $\bQ$ of genus $g\geq 1$. Choose in integer
$M\geq 1$ such that $C_{\bQ}$ has a a connected smooth proper model $C$ over $\bZ[1/M]$. Then $C$
has no $\Lambda$-structure. One can see this as follows.

If $g\geq 2$, let $p$ be a prime number such that $p\nmid M$. Then since $g\geq 2$, the map
$\psi_p$ cannot be a constant map on the fiber over $\bQ$ because it is not on the fiber over
$\bF_p$. Therefore it must be an automorphism on the fiber over $\bQ$. Further, there must exist an
integer $n\geq 1$ such that $\psi_p^{\circ n}$ is the identity on the fiber over $\bQ$. But as
$C_{\bQ}$ is dense in $C$, we see that $\psi_p^{\circ n}$ is the identity. This contradicts the
fact that $\psi_p$ is the Frobenius map on $C_{\bF_p}$, which is a nonempty curve.

If $g\geq 1$, then for any prime $p\nmid M$, there is a finite extension $K$ of $\bQ_p$ and point
$e\in C(K)=C(\sO_K)$, where $\sO_K$ denote the integral closure of $\bZ_p$ in $K$.
Let us now take the group law on $C_{\sO_K}$ to be the one for which $e$ is
the identity, and let $E$ denote the endomorphism ring $\End(C_{\bar\bQ_p})$. By enlarging $K$, we
may assume $E=\End(C_K)$. Note that $E$ is an integral domain of rank $1$ or $2$ over $\bZ$.

Then there is an element $\varphi\in E$ such that $\psi_p(x)=\psi_p(e)+\varphi(x)$, for all points
$x$. Therefore on each fiber of $C_{\sO_K}$ over $\sO_K$, the degree of $\psi_p$ is
$\varphi\bar\varphi\in\bZ$. But on the fiber over the residue field of $\sO_K$, the map $\psi_p$
agrees with the base-change of the $p$-th power Frobenius map, which has degree $p$. Therefore, we
have $\varphi\bar\varphi=p$. This rules out $E=\bZ$. To rule out the other case, observe that the
same equation implies $p$ is not inert in $E$. But since $p$ was allowed
to be any sufficiently large prime, this is impossible.

\section{Sub-$\Lambda$-spaces}
\label{sec:sub-lambda-spaces}

\subsection{} \emph{Periodic primes.}
\label{subsec:periodic-primes}
Let $X$ be a $\Lambda$-space.
Let us say that a prime number $p$ is \emph{periodic} if there exists an integer
$m\geq 1$ such that the endomorphism $\psi_p^{\circ m}$ of $X$ is the identity.
We also say that $m$ a \emph{period} of $p$. 
(Also see~Davydov~\cite{Davydov:periodic-lambda-rings}.)

\begin{proposition}\label{pro:finiteness-of-periodic-spaces}
	Let $X$ be a separated algebraic $\Lambda$-space of finite type over $\bZ$ 
	with infinitely many periodic primes. Then $X$ is affine and quasi-finite over $\bZ$.
\end{proposition}

\begin{proof}
	Let us consider quasi-finiteness first.  It suffices
	to assume $X$ is reduced and, hence, flat over $\bZ$.  
	(See~\cite{Borger:SLAG}.) \marpar{precise ref}
	Therefore it is enough to show that $X\times_{\Spec\bZ}\Spec\bQ$ is finite over 
	$\Spec \bQ$.
	
	Let $p$ be one of the given primes, and let $X_p$ denote the fiber of $X$ over $p$.
	Then $\psi_p$ is periodic, and therefore the $p$-th power Frobenius map on $X_p$
	is periodic.  By~\ref{lem:Frob-periodic-implies-discrete} below, the fiber $X_p$ is finite 
	over $\bF_p$.  Because there are infinitely many such $p$,
	$X$ is finite over $\bZ$ at a dense set of scheme-theoretic points.
	And because $X$ is of finite type, its fiber over $\bQ$ is finite.
	(See EGA IV (9.2.6.2)~\cite{EGA-no.28}.) 

	Affineness follows from quasi-finiteness.
	Since $X$ is separated,
	of finite type, and quasi-finite over $\bZ$, 
	Zariski's Main Theorem~\cite{EGA-no.11}, III 4.4.3, implies
	it is an open subscheme of a 
	scheme which is finite over $\bZ$, and any such scheme is affine.
	\marpar{Easier argument?}
\end{proof}

\begin{lemma}\label{lem:Frob-periodic-implies-discrete}
	Let $X$ be an algebraic space of finite type over $\bF_p$, and let
	$\Fr_X$ denote the $p$-th power Frobenius map on $X$.
	If $\Fr_X^{\circ n}$ is the identity map, then $X$ is a finite disjoint union of
	spaces of the form $\Spec F$, where $F$ is a field of degree at most $n$ over $\bF_p$.
\end{lemma}
\begin{proof}
	Let $U=\Spec B$ be an affine \'etale cover of $X$.  
	Then $\Fr^{\circ n}_U$ is the identity map on $U$.
	In particular, $B$ is reduced and is hence a subring of a finite product of fields.
	Because the $p$-th power map on $B$ has period $n$, the image in each field is a
	field of degree at most $n$ over $\bF_p$.  Since $U$ is a finite disjoint union
	of spectra of finite fields, and since $U$ covers $X$, $X$ is also such a space.
\end{proof}

\begin{proposition}\label{pro:finitely-many-maps-from-periodic-spaces}
 	Let $X$ and $Y$ be separated algebraic $\Lambda$-spaces of finite type over $\bZ$.
	Assume that $X$ is reduced and has infinitely many periodic primes.
	Then $\Hom_{\Lambda}(X,Y)$ is finite.
\end{proposition}
\begin{proof}
	Since $X$ is reduced, we can assume $Y$ is reduced.	
	We can also assume that $Y$ satisfies the same periodicity conditions as $X$.
	\marpar{more detail?}%
	Indeed, let $p$ be a periodic prime of $X$, and let $m_p$ denote its minimal period.
	Then the equalizer $Y_p$ of $\psi_p^{\circ m_p}$ and the identity
	map is a closed sub-$\Lambda$-space of $Y$.  Therefore, so is the
	reduced subspace of $Y'=\cap_p Y_p$, where $p$ runs over all periodic primes of $X$.
	But any map $X\to Y$ factors through $Y'$, which as a closed subspace of $Y$
	is of finite type over $\bZ$.
	Therefore it is enough to assume $Y=Y'$, which is to say that $Y$ satisfies the
	same periodicity conditions as $X$. 
	
	By~\ref{pro:finiteness-of-periodic-spaces}, there are generically finite rings $A$ and 
	$B$ such that $X=\Spec A$ and $Y=\Spec B$.
	Because $A$ and $B$ are reduced $\Lambda$-rings, they are torsion free, and thus
		$$
		\Hom_{\Lambda}(X,Y) = \Hom(B,A) \subseteq 
			\Hom(\bQ\tn_{\bZ}B,\bQ\tn_{\bZ}A).
		$$
	By Galois theory, there are only finitely many ring maps between 
	two finite \'etale algebras over a field; so $\Hom_{\Lambda}(X,Y)$ is finite.
\end{proof}

\begin{corollary}\label{cor:finitely-many-f1-points}
	Let $X$ be a separated algebraic space of finite type over $\bF_1$.
	Then there are only finitely many $\Lambda$-maps
	$\mu_n\to X$, where $\mu_n$ is defined in (\ref{subsec:monoid-algebras}).
	In particular, $X$ has only finitely many $\bF_1$-valued points.
\end{corollary}

\subsection{} \emph{Primitive $\Lambda$-spaces and complemented sub-$\Lambda$-spaces.} 
\label{subsec:complemented-points}
Let $f:Y\to X$ be a map of $\Lambda$-algebraic spaces which is a closed (resp.\ open)
immersion.  Then $Y$ (or better, $f$) is said to be \emph{complemented} if the complementary
open (resp.\ reduced closed) algebraic subspace $T$
admits a $\Lambda$-structure such that the map $T\to X$
is a $\Lambda$-map.  (Compare SGA 4 IV 9.1.13c~\cite{SGA4.1}.)
Note that because $T\to X$ is a monomorphism, 
the $\Lambda$-structure on $T$, when it exists, must be unique.

We also say that a $\Lambda$-space $X$ is \emph{primitive} if it is nonempty and its only
nonempty complemented closed sub-$\Lambda$-space is itself.

Observe that if $X'\to X$ is a $\Lambda$-morphism and $Y$ is a complemented closed (resp.\ open)
algebraic subspace of $X$, then its preimage $X'\times_X Y$ is a complemented closed (resp.\
open) algebraic subspace of $X'$. Also it is clear that finite intersections of complemented
closed (resp.\ open) $\Lambda$-algebraic subspaces are again complemented. Therefore the same is
true for finite unions.

For example, consider $\bA^1$ with the toric $\Lambda$-structure.
The $\bZ$-valued $\Lambda$-points of $\bA^1$ are in bijection with 
$\Hom_\Lambda(\bZ[x],\bZ)$, which agrees with 
	$$
	\setof{x\in\bZ}{x^p=x\,\text{for all primes $p$}} = \{0,1\}.
	$$
So $0$ and $1$ are the only $\bF_1$-valued points of the toric $\bA^1$.
The point $0$ is complemented, because its complement is the toric $\Gm$.
But the point $1$ is not complemented, because the preimage of $1$ under
$\psi_2$, say, is $\mu_2$, which is not contained in $\{1\}$.
More generally, the set of $\bF_1$-valued points of $\bA^d$ is $\{0,1\}^d$,
but the only complemented point is the origin.

\begin{proposition}\label{pro:tori-are-primitive}
	The $\Lambda$-space $\Gm^d$ is primitive, for any integer $d\geq 0$.
\end{proposition}
\begin{proof}
	Let $Z$ be a nonempty
	complemented closed sub-$\Lambda$-space of $\Gm^d$.
	Since $\red{Z}$ is flat over $\bZ$ (\cite{Borger:SLAG}) and nonempty, the
	space $Z$ has a $\bC$-valued point $z$.
	Write $z=(e^{w_1},\dots,e^{w_d})$ for some numbers $w_1,\dots,w_d\in\bC$.
	For each integer $s\geq 1$, the point
		$$
		z_s = (e^{w_1/s},\dots,e^{w_d/s})
		$$
	satisfies $\psi_s(z_s)=z\in Z$.  Because $Z$ is complemented, we then
	have $z_s\in Z$.  But in the analytic topology we have
		$$
		\lim_{s\to\infty} z_s = (1,\dots,1).
		$$
	Since $Z$ is closed (in the Zariski and hence analytic topology), it contains $(1,\dots,1)$,
	and because it is complemented, it must contain $\psi_n^{-1}(1,\dots,1)=\mu_n^d$
	for any integer $n\geq 1$. 
	On the other hand, $\cup_n\mu_n$ is Zariski dense in $\Gm$.
	Therefore $\cup_n\mu_n^d$ is Zariski dense in $\Gm^d$, and so
	we have $Z=\Gm^d$.  
\end{proof}

\begin{proposition}\label{pro:comp-closed-of-torics}
	Let $X$ be a toric variety.  Then a closed sub-$\Lambda$-space
	is complemented if and only if it is a union of closures of torus orbits.
\end{proposition}

For background on toric varieties, see Fulton's notes~\cite{Fulton:Toric-book}, 
especially sections 2.1 and 3.2.

\begin{proof}
	Let $Z$ be a closed sub-$\Lambda$-space of $X$.
	Suppose $Z$ is a union of closures $Z_i$ of torus orbits.
	Then each $Z_i$ is the toric subvariety
	\marpar{why exactly?}
	corresponding to a fan, and is therefore complemented.  
	Since there are only finitely many torus orbits, the union
	$Z$ of the $Z_i$ is complemented.

	Now assume instead that $Z$ is complemented.  Then its intersection with any torus
	orbit $Y$ is either $Y$ or $\emptyset$, by~\ref{pro:tori-are-primitive}.
	Therefore $Z$ is a union of torus orbits.  
	Since $Z$ is closed, it is also a union of their closures.
\end{proof}

\begin{corollary}\label{cor:comp-points-of-toric-varieties}
	The complemented $\bF_1$-points of a toric variety $X$ are the fixed points
	of the torus action. In particular, the number of complemented $\bF_1$-points
	is the Euler characteristic. 
	\marpar{write out proof?}
\end{corollary}

For example, the only complemented $\bF_1$-valued point of the toric $\bA^n$ is the origin. The
toric $\bP^n$ has exactly $n+1$ complemented points with values in $\bF_1$. They are
$[1,0,\dots,0]$, $[0,1,0,\dots,0]$, and $[0,\dots,0,1]$. These facts have been predicted in
earlier speculations on the field with one element~\cite{Tits:F1}, p.\ 285.

\section{Function spaces}
\label{sec:function-spaces}

In this section, we discuss $\GL_n$ over $\bF_1$. The group scheme $\GL_n$ does not descend to
$\bF_1$. (See Buium~\cite{Buium:differential-subgroups}.) But we can realize $\GL_n$ as the
automorphism group of something that does descend to $\bF_1$. This might be surprising, because
formation of function spaces commutes with base change. Indeed, that is one important
way in which the topos map $\ltm\:\Space_{\bZ}\to\Space_{\bF_1}$ is not like a true map of spaces.

For discussion of $\GL_n$ in other approaches to $\bF_1$, see for example
Connes--Consani~\cite{Connes-Consani:notion-of-geometry-over-f1} and
To\"en--Vaqui\'e~\cite{Toen-Vaquie:Under-Spec-Z}.

This section was directly inspired by some questions Kirsten Wickelgren asked me.

\subsection{}
Let $X$ and $Y$ be objects of $\Space_{\bF_1}$.  
Write $\Hom_{\bF_1}(X,Y)$ for the set of maps from $X$ to $Y$, 
and write $\Hom_{\bZ}(X,Y)$ for the set of maps $\ltm^*X\to\ltm^*Y$ 
between the underlying objects of $\Space_{\bZ}$.  Let $\sHom_{\bF_1}(X,Y)$ denote the 
usual $\Space_{\bF_1}$-object of maps from $X$ to $Y$.  It is defined by
	$$
	\sHom_{\bF_1}(X,Y)\: T \mapsto \Hom_{T}(X\times T, Y\times T),
	$$
for any space $T\in\Space_{\bF_1}$. 
We define $\sHom_{\bZ}(X,Y)$ similarly.

Then we have a map
	\begin{equation} \label{eq:lambda-locus-map-1}
		\sHom_{\bF_1}(X,Y) \longmap \ltm_*(\sHom_{\bZ}(X,Y)),
	\end{equation}
which sends an $\bF_1$-map $a\:X\times T\to Y\times T$ to its underlying $\bZ$-map 
$\ltm^*(a)$.  
This map is clearly injective.

Adjunction then gives another map
	\begin{equation} \label{eq:lambda-locus-map-2}
			\ltm^*\:\ltm^*(\sHom_{\bF_1}(X,Y)) \longmap \sHom_{\bZ}(X,Y).
	\end{equation}
Note that this map is generally neither a monomorphism nor an epimorphism.
This is one way in which the topos map $\ltm\:\Space_{\bZ}\to\Space_{\bF_1}$ is different
from one induced by a true map of spaces.
For example, if $X=W(\Spec\bQ)$, then~(\ref{eq:lambda-locus-map-2}) is identified with
the map $Y^{\bZ}\to Y^{\bN}$, which is rarely a monomorphism.
On the other hand, if $X=Y=\bA^1_{\bF_1}$, then the map is not an epimorphism. 
\marpartd{True? Better\\example? Check}

Recall that an endomorphism $\bA^n\to\bA^n$ is a linear transformation if and only if 
it is equivariant under the 
action of $\Gm$ on $\bA^n$ given by scalar multiplication:
	$$
	\bZ[x_1,\dots,x_n] \longlabelmap{x_i\mapsto x_i\tn z} 
		\bZ[x_1,\dots,x_n]\tn_\bZ \bZ[z,z^{-1}].
	$$
Now observe that if we give $\Gm$ and $\bA^n$ their toric $\Lambda$-structures, then
this action morphism is $\Lambda$-equivariant.  Therefore it is reasonable to define
	\begin{equation}
		M_{n/\bF_1} = \sHom_{\Gm/\bF_1}(\bA^n,\bA^n),
	\end{equation}
where for general $\bF_1$-spaces $X,Y$ with $\Gm$-actions, we
define  $\sHom_{\Gm/\bF_1}(X,Y)$ to be the sub-$\bF_1$-space of 
$\sHom_{\bF_1}(X,Y)$ 
consisting of the $\Gm$-equivariant maps.
Let us also set 
	\begin{equation}
		\GL_{n/\bF_1} = \sAut_{\Gm/\bF_1}(\bA^n),		
	\end{equation}
to be locus of $M_{n/\bF_1}$ consisting of invertible maps.
Of course, $M_{n/\bF_1}$ is a monoid object in $\Space_{\bF_1}$ and 
$\GL_{n/\bF_1}$ is a group object.
In~\ref{cor:GL-over-F1} below, we will describe them concretely.

Note that because $\ltm^*$ is faithful,
we can view $\ltm^*\big(\sHom_{\Gm/\bF_1}(X,Y)\big)$
as a subspace of $\sHom_{\Gm}(\ltm^*X,\ltm^*Y)$.
Consider this when $X=\bA^n$ and $Y=\bA^1$, both with the toric $\Lambda$-structures.
Then for any ring $B$, a $B$-valued point of $\Hom_{\Gm}(\bA^n,\bA^1)$ is just a 
$B$-module homomorphism $B^n\to B$, or simply a $1\times n$ matrix $(b_1,\dots,b_n)$
with entries in $B$.  This furnishes an identification
	\begin{equation} \label{eq:gln-ident}
		\Hom_{\Gm}(\bA^n,\bA^1)=\Spec\bZ[b_1,\dots,b_n]=\bA^n
	\end{equation}

\begin{proposition}\label{pro:functionals-over-f1}
	Under the identification~(\ref{eq:gln-ident}), the subspace 
	of $\bA^n$ corresponding to the subspace
	$\sHom_{\Gm/\bF_1}(\bA^n,\bA^1)$ of $\sHom_{\Gm}(\bA^n,\bA^1)$ 
	is the union of the axes
		$$
		Z=\Spec\bZ[b_1,\dots,b_n]/(b_ib_j:i\neq j).
		$$
	Further the operators on $Z$ induced by $\psi_p$ take each coordinate $b_i$ to $b_i^p$.
\end{proposition}
\begin{proof}
	We just calculate the $B$-valued points of $\ltm^*\sHom_{\Gm/\bF_1}(\bA^n,\bA^1)$
	for any ring $B$.  We have
		$$
		(\ltm^*\sHom_{\Gm/\bF_1}(\bA^n,\bA^1))(\Spec B) = 
			(\sHom_{\Gm/\bF_1}(\bA^n,\bA^1))(\ltm_!(\Spec B)).
		$$
	But since $\sHom_{\Gm/\bF_1}(\bA^n,\bA^1)$ sits inside
	$\sHom_{\Gm}(\bA^n,\bA^1)$, which is affine, 
	any map $\ltm_!(\Spec B)\to \sHom_{\Gm/\bF_1}(\bA^n,\bA^1)$
	factors through the affinization $\Spec W(B)$ of $\ltm_!(\Spec B)$.
	Therefore we have
		$$
		(\ltm^*\sHom_{\Gm/\bF_1}(\bA^n,\bA^1))(\Spec B) = 
			(\sHom_{\Gm/\bF_1}(\bA^n,\bA^1))(\Spec W(B)).
		$$
	If we think of points of $\bA^n$ as $n$-dimensional column vectors, then 
	it is natural to think of the points 
	of $\sHom_{\Gm}(\bA^n,\bA^1)$ as $1\times n$ matrices.
	Given such a matrix
		$$
		(a_1,\dots,a_n)\in W(B)^n,
		$$
	the corresponding map in $\sHom_{\Gm}(\bA^n,\bA^1)(\Spec W(B))$ is
	an $\bF_1$-map if and only if the ring map
	\begin{equation}\label{eq:gln-map3}
		\varphi\:\bZ[t] \longmap \bZ[x_1,\dots,x_n]\tn W(B)
	\end{equation}
	determined by $t\mapsto \sum_j x_j\tn a_j$
	is a $\Lambda$-ring map, where $\bZ[x_1,\dots,x_n]$ has the toric $\Lambda$-structure.

	Thus to finish the proof, it suffices to
	show that necessary and sufficient conditions for $\varphi$ to be
	\marpar{why exactly?}
	a $\Lambda$-ring map are,
	first, that each element $a_j$ is the Teichm\"uller lift of some element $b_j\in B$ and,
	second, that $b_ib_j=0$ for all $i\neq j$.  
	
	Let us first show the necessity.  So assume $\varphi$ is a $\Lambda$-ring map.
	Consider, for each $i$, the map
		$$
		\pi_i\:\bZ[x_1,\dots,x_n]\tn W(B) \longmap W(B)
		$$
	given by $x_j\mapsto \delta_{ij}$ and by the identity on $W(B)$.
	Then $\pi_i\circ\varphi$ is a $\Lambda$-ring map under which the image of $t$ is $a_i$.
	
	It follows that $a_{i}$ 
	must be a Teichm\"uller lift.  Indeed, giving
	a $\Lambda$-map $\bZ[t] \to W(B)$ is by adjunction the same as giving a ring map 
	$\bZ[t]\to B$.  This in turn is the same as an element of $b_i\in B$.
	Tracing through these identifications shows that $a_i=[b_i]$, and thus
	the necessity of the first condition above.
	
	To see the necessity of the second condition, 
	use the operator $\lambda_2\in\Lambda$:
		$$
		0 = \varphi(0)=\varphi(\lambda_2(t))=\lambda_2(\varphi(t))
			= \lambda_2\Big(\sum_j x_j\tn [b_j]\Big)
			= \sum_{i<j} x_i x_j\tn [b_i][b_j].
		$$
	Thus for $i\neq j$, we have $[b_i b_j]=[b_i][b_j]=0$ and hence $b_ib_j=0$.
	
	Sufficiency is similar.  Suppose $\varphi$ is defined by 
		$$
		\varphi(t) = \sum_j x_j\tn[b_j]
		$$
	with $b_i b_j=0$ for all $i\neq j$.  To show $\varphi$ is a $\Lambda$-ring map,
	it is enough to show it commutes with $\lambda_n$ for $n\geq 2$.  But we
	have that $\lambda_n(x_j\tn[b_j])=0$ for all $n\geq 2$.   
	Therefore $\lambda_n(\sum_j x_j\tn[b_j])$ is the $n$-th elementary symmetric 
	function in the $x_j\tn [b_j]$.  Because $b_ib_j=0$ for $i\neq j$, each of
	the elementary monomials is zero when $n\geq 2$.  Therefore we have
		$$
		\lambda_n(\varphi(t)) = \lambda_n\Big(\sum_j x_j\tn [b_j]\Big) = \lambda_n(0) 
			= \varphi(0) = \varphi(\lambda_n(t)),
		$$
	for $n\geq 2$.  Thus $\varphi$ is a map of $\Lambda$-rings.
\end{proof}

\begin{corollary}\label{cor:GL-over-F1}
	We have the following equalities of
	subspaces of $M_n$:
		\begin{equation}
			M_{n/\bF_1}=Z^n \quad\quad\text{and}\quad\quad\GL_{n/\bF_1} = S_n\ltimes \Gm^n.
		\end{equation}
\end{corollary}
\begin{proof}
	The first equality follows from the universal property of products 
	and~\ref{pro:functionals-over-f1}:
		$$
		\sHom_{\Gm/\bF_1}(\bA^n,\bA^n) = \big(\sHom_{\Gm/\bF_1}(\bA^n,\bA^1)\big)^n
			= Z^n.
		$$
	For the second equality, recall that a morphism of $\Lambda$-spaces is
	an isomorphism if and only if it becomes one after applying $\ltm^*$.
	Therefore we have
		$$
		\GL_{n/\bF_1} = \GL_n \cap\; M_{n/\bF_1} = S_n\ltimes\Gm^n.
		$$
\end{proof}

\begin{corollary}\label{cor:F1-points-of-GLn}
	\begin{enumerate}
		\item 	$M_{n/\bF_1}(\bF_1)$ is the set of $n\times n$ matrices with the property
			that every entry is either $0$ or $1$ and 
			every row has at most one $1$.
		\item 	$\GL_{n/\bF_1}(\bF_1)=S_n$.
	\end{enumerate}
\end{corollary}

\subsection{} \emph{Remarks.}
Note that none of the $\bF_1$-valued points of $\GL_{n/\bF_1}$ are
complemented.  Also note that the determinant map $\GL_{n/\bF_1}\to \Gm$ is not
a $\Lambda$-map, because it fails to commute with $\psi_2$.  It does
commute with the other $\psi_p$ though.  

\tempcomment{
\vspace{3mm}
\marpar{temp comment}
TO DO
\begin{enumerate}
	\item Do the same thing for flag varieties?  Define Grassmannians by 
		$\sHom^{\Lambda}_{\Gm}(\bA^d,\bA^n)$.
	\item $PGL_n^{\Lambda}$ exists
	\item Other algebraic groups from root systems?  (Mark's suggestion.)
\end{enumerate}}

\section{Aside: abelian motives are potentially cyclotomic}
\label{sec:abelian-motives}

Let $p$ be a prime number. The purpose of this section is to establish the result
(\ref{thm:abelian-motives-are-Artin-Tate}) in usual, non-$\Lambda$ algebraic geometry that
abelian $p$-adic Galois representations of geometric origin are Artin--Tate. The proof is an
result in $p$-adic Hodge theory.  It is well within the scope of established techniques,
but to my knowledge, it is not actually in the literature. For results of a similar flavor, see
Wang~\cite{Wang-CL:Cohomology-theory}, Kisin--Lehrer~\cite{Kisin-Lehrer:CountingPonts}, and van
dan Bogaart--Edixhoven~\cite{vandenBogaart-Edixhoven}.\footnote{The style of this section
is somewhat clumsy.  I hope to improve it in a future version.}

\begin{proposition}\label{pro:cohomology-is-constructible}
	Let $X$ be a separated scheme of finite type over $\bZ[1/p]$.
	Let $f\:X\to \Spec\bZ[1/p]$ denote the structure map.
	Then $R^nf_!(\bQ_p)$ is lisse away from a finite set of primes.
	If $X$ is smooth and proper over $\bZ[1/M]$, then the sheaf is lisse
	at all primes not dividing $M$.
\end{proposition}

Recall that we can define cohomology with compact support because, by 
Nagata's theorem, any separated morphism
of finite type between noetherian schemes is compactifiable.  (See 
Conrad~\cite{Conrad:Nagata-compact}, theorem 4.1.  Actually, according to a forthcoming
paper of Conrad--Lieblich--Olsson, the scheme-theoretic hypotheses can be removed.)

\begin{proof}
	The sheaf $R^nf_!(\bZ/p\bZ)$ is constructible (Deligne SGA 4 1/2~\cite{Deligne:SGA4.5}
	[Arcata] IV (6.2))
	and hence lisse away from a finite set $T$ of primes.
	For $m\geq 0$, if $R^nf_!(\bZ/p^m\bZ)$ is lisse, then
	all subsheaves and quotient sheaves of $R^nf_!(\bZ/p^m\bZ)$
	are lisse away from $T$.  On the other hand, by
	the long exact sequence of cohomology, $R^nf_!(\bZ/p^{m+1}\bZ)$
	is an extension of a subsheaf of $R^nf_!(\bZ/p\bZ)$ by a quotient
	sheaf of $R^nf_!(\bZ/p^{m}\bZ)$ and is therefore 
	also lisse away from $T$.  By induction $R^nf_!(\bZ/p^m\bZ)$ is lisse away
	from $T$ for all $m$.
	Therefore $R^nf_!(\bZ_p)$ and hence $R^nf_!(\bQ_p)$ are lisse away from $T$.
	
	The final statement follows from Deligne, SGA 4 1/2~\cite{Deligne:SGA4.5} 
	[Arcata] V (3.1).
\end{proof}

\subsection{} \emph{Cyclotomic representations.}
Let $E$ be a finite extension of $\bQ_p$.
Let $K$ be a finite extension of $\bQ$ or $\bQ_p$, 
and let $V$ be a $d$-dimensional continuous $E$-linear representation of
$\Gal(\bar{K}/K)$. Let us say that $V$ is \emph{cyclotomic} if there exist integers
$j_1,\dots,j_d$ such that $V$ is isomorphic to $E(j_1)\oplus\dots\oplus E(j_d)$.
Let us say that $V$ is \emph{potentially cyclotomic} if there exists a finite extension
$L$ of $K$ such that $V$ is cyclotomic as a representation of $\Gal(\bar{K}/L)$.

Observe that these concepts are independent of $E$ in the sense
that for any finite extension $E'$ of $E$, we have
\begin{equation} \label{eq:cyclotomic-reps-and-coeffs}
	E'\tn_E V \text{ is (potentially) cyclotomic if and only if $V$ is.}
\end{equation}

\begin{theorem}\label{thm:abelian-motives-are-Artin-Tate}
	Let $X$ be a separated scheme of finite type over $\bQ$.
	Suppose that $\Gal(\qbar/\bQ)$ acts on $\Hetc^n(X_{\qbar},\bQ_p)$ 
	through its abelianization.
	Then there is an integer $N\geq 1$ such that the restriction of
	$\Hetc^n(X_{\qbar},\bQ_p)$ to $\Gal(\bar\bQ/\bQ(\zeta_N))$ is  cyclotomic.
	
	If $X$ has a smooth and proper model over $\bZ[1/M]$,
	then $N$ can be taken such that all its prime divisors are divisors of $Mp$.
	\marpartd{prove 2nd part}
\end{theorem}

\marpar{lots of incon\\sistent notation}
\begin{proof}
	By the main theorem of $p$-adic Hodge theory, $\Hetc^n(X_{\qbar},\bQ_p)$
	is a potentially semi-stable representation of $\Gal(\bar\bQ_p/\bQ_p)$.
	(This is the work of many people.  See~\cite{Kisin:pst} for the 
	final form of the theorem.)  By~\ref{pro:cohomology-is-constructible},
	the set of
	ramified primes is finite, and if $X$ has a smooth proper
	model over $\bZ[1/M]$, it contains only primes dividing $Mp$.
	The theorem is then an immediate consequence of~\ref{thm:abstract-abelian-Galois}. 
%
\end{proof}

\begin{theorem}\label{thm:abstract-abelian-Galois}
	Let $E$ be a finite extension of $\bQ_p$.
	Let $\shfcond$ be a finite set of prime numbers containing $p$, and
	let $V$ be a finite-dimensional representation of $G_{\bQ}$ satisfying
	the following properties:
		\begin{enumerate}
			\item $G_{\bQ}$ acts on $V$ through its abelianization,
			\item $V$ is unramified away from $\shfcond$,
			\item $V$ is potentially semi-stable at $p$.
		\end{enumerate}
	Then there is an integer $N$ divisible only by the primes in $\shfcond$
	such that the restriction of
	$V$ to $\Gal(\bar\bQ/\bQ(\zeta_N))$ is  cyclotomic.
\end{theorem}

\begin{proof}
	By the Kronecker--Weber theorem, there is an integer $n\geq 1$, divisible only by
	primes in $\shfcond$, such that the action of $G_{\bQ,S}$ on $V$ factors through
	the cyclotomic character $G_{\bQ,S}\to (\bZ/n\bZ)^*\times \bZ_p^*$.  
	Let $H$ be the torsion subgroup of $(\bZ/n\bZ)^*\times \bZ_p^*$.
	Let us show that it is enough to consider the case where $H$ acts trivially on $V$.

	Let $G$ denote the quotient $(\bZ/n\bZ)^*\times \bZ_p^*/H$.  
	Since $G\cong\bZ_p$, we have
		$$
		(\bZ/n\bZ)^*\times \bZ_p^* \cong G\times H.
		$$
	By~(\ref{eq:cyclotomic-reps-and-coeffs}), we can assume every irreducible $E$-linear
	representation of $H$ is one-dimensional.  For each character $\rho:H\to E^*$,
	let $V_{\rho}$ denote the summand of $V$ on which $H$ acts via $\rho$.
	Since $G\times H$ is abelian, $V_{\rho}$ is stable under $G\times H$.
	Therefore we have a decomposition of Galois representations
		$$
		V = \bigoplus_{\rho} V_{\rho}.
		$$
	It is enough to show each $V_{\rho}$ is potentially cyclotomic.  On the
	other hand, because
	each $V_{\rho}$ satisfies (a)--(c), it is enough to assume
	$V=V_{\rho}$ for some $\rho$.	
	Now observe that $\rho$, viewed as a Galois representation by the composition
		$$
		G\times H \longmap H \longlabelmap{\rho} E^*,
		$$
	is potentially cyclotomic, since it is potentially trivial.
	Therefore it is enough to show $V\tn\rho^{-1}$
	is potentially cyclotomic.  Similarly, $\rho^{-1}$ satisfies (a)--(c) above,
	and hence so does $V\tn\rho^{-1}$.  But $H$ acts trivially on $V\tn\rho^{-1}$.
	Therefore it is indeed enough to consider the case where $H$ acts trivially on $V$.

	Let $g$ be a pro-generator of $G/H$.
	By~(\ref{eq:cyclotomic-reps-and-coeffs}), 
	it is sufficient to assume that $E$ contains all the eigenvalues of $g$
	acting on $V$.  Then $V$ has a
	$g$-stable filtration 
		$$
		0=V_0\subseteq V_1\subseteq \cdots  \subseteq V_{d} = V
		$$ 
	by sub-$E$-vector spaces such that
	$\dim_E(V_{i}/V_{i-1})= 1$, for $i=1,\dots,d$.  Because $g$ is a pro-generator
	of $G/H$, the filtration is $G/H$-stable.
	
	Because $V$ is potentially semi-stable, so is each subquotient $V_i/V_{i-1}$.
	For each $i$, the lemma~\ref{lem:try-2} implies
	there is finite extension $K_i$ of $\bQ_p$ in 
	$\bQ_p(\zeta_{p^\infty})$ such that
	the action of $\Gal(\bQ_p(\zeta_{p^\infty})/K_i)$ 
	on $V_{i}/V_{i-1}$ is an integral power of the cyclotomic character.
	Take an integer $r\geq 0$ such that $\bQ_p(\zeta_{p^r})$ contains $K_1,\dots,K_d$.
	Then the action of $\Gal(\bQ(\zeta_{p^\infty})/\bQ_p(\zeta_{p^r}))$ on
	$V_{i}/V_{i-1}$ is an integral power of the cyclotomic character.
	In other words, $V$ is an iterated extension of 
	$\Gal(\bQ(\zeta_{p^\infty})/\bQ_p(\zeta_{p^r}))$-representations
	of the form $E(j)$.  
	But by~\ref{lem:lemma2}, it must then be a direct sum of
	such representations.  Thus it is cyclotomic as a representation of 
	$\Gal(\bQ(\zeta_{p^\infty})/\bQ_p(\zeta_{p^r}))$.
	But because the natural map
		$$
		\Gal(\bQ_p(\zeta_{p^\infty})/\bQ_p(\zeta_{p^r})) \longmap
			\Gal(\bQ(\zeta_{p^\infty})/\bQ(\zeta_{p^r}))
		$$
	is an isomorphism, $V$ is  cyclotomic as a representation of $G_{\bQ(\zeta_{p^r})}$.
\end{proof}

\begin{lemma}\label{lem:try-2}
	Let $E$ be a finite extension of $\bQ_p$, and let $V$ be
	a one-dimensional $E$-vector space with a continuous $E$-linear action of
	$\Gal(\bQ_p(\zeta_{p^\infty})/\bQ_p)$.
	Assume further that $V$ is potentially semi-stable (as a $\bQ_p$-representation of 
	$G_{\bQ_p}$).  Then there is a finite extension $K$ of $\bQ_p$ such that
	for some integer $i$, we have $V\cong E(i)$ as representations of 
	$\Gal(K(\zeta_{p^\infty})/K)$.
\end{lemma}
\begin{proof}
	Let $d=[E:\bQ_p]$.
	Let $M$ be the weakly admissible module associated to $V$.
	Then $M$ inherits from $V$ an action of $E$.  That is, 
	every $e\in E$ acts as a morphism of weakly admissible modules.
	Therefore the Hodge
	filtration on $M$ is a filtration by sub-$E$-modules
	and the Frobenius and monodromy operators are $E$-linear.
	Because $\dim_E(M)=\dim_{\bQ_p}(M)/d=1$,
 	there is an integer $j$ such that $F^jM = M$ and $F^{j+1}M=0$, and
	there is an element $\alpha\in E$ such that
	the endomorphism $\varphi$ of $M$ is multiplication by $\alpha$.
	Therefore the the Hodge number of $\det_{\bQ_p}(M)$ is $dj$ and the slope 
	is $d v_{\bQ_p}(\alpha)$.
	Because $M$ is weakly admissible, these two must be equal,
	and so we have $v_{\bQ_p}(\alpha) = j$.  Replacing $M$ by
	$M(-j)$, it suffices to assume $j=0$.  We now want to show
	that $M$ becomes trivial after a finite extension.
	
	Let $k$ denote the residue field of $K$.
	Let $M'=M\tn_{W(k)} W(\bar{k})$.  Then $M'$ is a one-dimensional potentially
	semi-stable weakly admissible $E$-module over $W(\bar{k})$ with $j=0$.
	Let us show that $M'$ is isomorphic to the trivial weakly admissible module.
	
	Let us now show that we can change basis of $M'$ so that $\alpha=1$.  
	(This is a consequence of Manin's theorem when $E=\bQ_p$.)
	The ring $E\tn W(\bar{k})$ has an endomorphism $\id\tn \sigma$.
	The ratio of this with the identity map gives a group endomorphism $f$ of
	$(E\tn W(\bar{k}))^*$ sending $a\tn x\mapsto a\tn\sigma(x)/x$.
	We need to show this is a surjection.  
	Define a filtration $A^i$ of $(E\tn W(\bar{k}))^*$ by setting $A^i$
		$$
		(A\tn W(\bar{k}))^* \longmap (E/\m^i\tn W(\bar{k}))^*,
		$$
	where $\m$ is the maximal ideal of $E$.
	Because the ring $E\tn W(\bar{k})$ is
	$\m$-adically complete, the group $(E\tn W(\bar{k}))^*$ is complete
	with respect to this filtration.  Therefore it is enough to show $f$ is
	surjective on the associated graded abelian group.  We have
	$\gr^0 A = (E/\m \tn \bar{k})^*$ and for $i\geq 1$,
	we have $\gr^i A\equiv (\m^i/\m^{i+1} \tn \bar{k})$ by the map $x\mapsto x-1$.
	The map $f$ becomes $x\mapsto x^{p-1}$ on $\gr^0 A$, and it becomes
	$a\tn x\mapsto a\tn (x^p-x)$ on $\gr^iA$ for $i\geq 1$.  Because $\bar{k}$ is
	algebraically closed, both these maps are surjective.
	
	Let us now show
	that the monodromy operator $N$ is zero.
	But this holds because $\varphi$ and $N$ are two $E$-linear endomorphisms
	of a one-dimensional vector space satisfying $\varphi N = p N\varphi$
	and $\varphi\neq 0$.  And so the weakly admissible module associated
	to the $G_{\bQ\tn W(\bar{k})}$-representation $V$ is trivial.
	
	Therefore there is a finite extension $L$ of $\bQ\tn W(\bar{k})$ such that
	$G_L$ acts trivially on $V$.  Therefore there is a finite extension $K$
	of $\bQ_p$ such that the inertia group $I_K$ of $G_K$ acts trivially on $V$.
	Since the action of $G_{\bQ_p}$ on $V$ factors through 
	$\Gal(\bQ_p(\zeta_{p^\infty})/\bQ_p)$, the action of $G_K$ .
\end{proof}

\begin{lemma}\label{lem:lemma2}
	Let $E$ and $K$ be finite extensions of $\bQ_p$, and
	let $V$ be a finite-dimensional continuous $E$-representation of 
	$\Gal(K(\zeta_{p^\infty})/K(\zeta_{p^b}))$.
	If $V$ is Hodge--Tate and
	has a semi-simplification which is isomorphic to a sum of cyclotomic representations,
	then $V$ itself is isomorphic to a sum of cyclotomic representations.
\end{lemma}
\begin{proof}
	Let $G$ denote $\Gal(K(\zeta_{p^\infty})/K(\zeta_{p^b}))$.
	Giving a representation of $G$ is the
	same as giving a representation of its Lie algebra, which is one-dimensional;
	and so this is the same as giving a matrix $\theta$.  We have assumed that
	$\theta$ is upper-triangular with only integers on the diagonal.
	One can change basis to make $\theta$ a block-diagonal matrix, where each
	block is upper-triangular and has a single integer on the diagonal.
	Therefore it is enough to assume $\theta$ has is upper triangular and has
	a single integer on the diagonal.  Twisting by a cyclotomic character, we
	can assume $\theta$ is nilpotent.
	
	Thus it suffices to assume that $V$ has a trivial semi-simplification.
	By induction on $\dim_E(V)$, we need only prove that any Hodge--Tate extension $W$ of
	the trivial representation by itself is split.  Therefore the representation
	is given by a group map $G\to E$ into the upper-right corner
	of the matrix.  Since the Hodge--Tate weights
	of $W$ are both $0$, it must in fact be $\bC_p$-admissible.  Therefore by Sen's
	theorem, 
	it factors through finite quotient of $G$.  But the only 
	group map $G\to E$ with this property is the trivial map.  Therefore $V$ 
	is a split extension.
\end{proof}

\section{$p$-adic \'etale cohomology}
\label{sec:p-adic-cohomology}

Let $X$ be a separated $\Lambda$-scheme of finite type over $\bZ$.
Let $r_{m,n}$ denote the Hodge number $h^{m,m}$ of $\Hc^n(X^{\an},\bC)$,
where $X^{\an}$ is the complex-analytic space underlying $X$.
(See Deligne~\cite{Deligne:HodgeII}, (2.3.7), for the definition of the Hodge numbers
of a mixed Hodge structure and Deligne~\cite{Deligne:HodgeI} \marpar{correct ref?}
for the fact that $\Hc^n(X^{\an},\bC)$
carries a natural mixed Hodge structure.)
Let us write
		$$
		P(t) = \sum_{m,n} (-1)^n r_{m,n}t^m \in \bZ[t].
		$$
Finally, let us fix an algebraic closure $\bar\bQ$ of $\bQ$.

\begin{theorem}\label{thm:lambda-varieties-are-abelian}
	For any integer $\coeff>0$, the action of $\Gal(\qbar/\bQ)$
	on $\Hetc^n(X_{\qbar},\bZ/\coeff\bZ)$ factors through $\Gal(\qbar/\bQ)^{\ab}$.
\end{theorem}
\begin{proof}
	By Deligne~\cite{Deligne:SGA4.5}, [Arcata] IV (6.2), 
	the sheaf $R^nf_!(\bZ/\coeff\bZ)$ is constructible.
\marpar{clean up}
	Therefore, there exists a finite Galois extension $K/\bQ$ and an integer $N>0$
	such that the restriction of $R^nf_!(\bZ/\coeff\bZ)$ to $\Spec \sO_K[1/N]$ is
	the constant sheaf associated to an abelian group $V$.
	Here, $\sO_K$ denotes the ring of integers of $K$.
	By functoriality $V$ has an action of $G=\Gal(K/\bQ)$.
	Let us also assume that $N$ is a multiple of the discriminant of $K$.
	By the base-change theorem, we have an isomorphism
		$$
		\Hetc^n(X_{\qbar},\bZ/\coeff\bZ)\cong V 
		$$
	of representation of $\Gal(\bar\bQ/\bQ)$ (which acts on $V$ via the map to $G$).
	Therefore it is enough to show the action of $G$ on $V$ is through its abelianization.
	
	Let $p$ be a prime not dividing $N$.  Let $\fp$ be a prime of $K$ over $p$,
	let $k_\fp$ denote $\sO_K/\fp$, and let $D$ denote the decomposition subgroup
	of $G$ corresponding to $\fp$.  Because $p\nmid N$, 
	the map $D\to\Gal(k_\fp/\bF_p)$ is an isomorphism.
	By the proper base-change theorem, we have an isomorphism
		$$
		V \cong\Hetc^n(X_{\bar{\bF}_p},\bZ/\coeff\bZ)
		$$
	of representations of $\Gal(\bar\bF_p/\bF_p)$.
	Further, the endomorphism $\psi_p^*$ of $V$ 
	(induced by the endomorphism $\psi_p$ of $X$) corresponds
	to the endomorphism $\Fr_p^*$ of the right-hand side.
	By generalities about the Frobenius map, $\Fr_p^*$ acts on 
	$\Hetc^n(X_{\bar{\bF}_p},\bZ/\coeff\bZ)$ as the inverse of the
	residual arithmetic Frobenius element $\Frob_p\in D/I$, the automorphism defined by
	$\Frob_p(x)=x^p$.
	Therefore $\psi_p^*$ acts on $V$ in the same way as $\Frob^{-1}_D$
	where $\Frob_D$ is the element of $D$ mapping to $\Frob_p$.
	
	On the other hand, the endomorphisms $\psi_p$ of $X$
	commute with each other as $p$ varies.  Therefore the endomorphisms $\psi_p^*$ of
	$\Hetc^n(X_{\qbar},\bZ/\coeff\bZ)$ commute with each other.
	By the above, the action of any Frobenius elements $\Frob_D$ for $p\nmid N$, 
	commute with each other.  By Chebotarev's theorem 
	(see Neukirch~\cite{Neukirch:CFT}, V (6.4)), every element of $G$
	is such a Frobenius element.  Therefore $G$ acts on $V$ through its abelianization.
\end{proof}

\begin{corollary}\label{cor:lambda-varieties-are-Artin-Tate}
	For any prime number $p$, there is an integer $N\geq 1$ 
	such that there is an isomorphism
	of representations of $\Gal(\qbar/\bQ(\zeta_N))$:
	\begin{equation} \label{eq:sum-of-tate-twists}
		\Hetc^n(X_{\qbar},\bQ_p) \cong \bigoplus_m \bQ_p(-m)^{r_{m,n}}.
	\end{equation}
	If $X$ is smooth and proper over $\bZ[1/M]$, for some integer $M>0$,
	then $N$ can be taken 
	such that all its prime divisors are divisors of $Mp$.
\end{corollary}

\marpar{Is there a minimal\\conductor?}
Let us call an integer $N>0$ satisfying the conclusion
of~\ref{cor:lambda-varieties-are-Artin-Tate} a \emph{conductor} of $X$.

\begin{proof}
	By~\ref{thm:lambda-varieties-are-abelian} 
	and~\ref{thm:abelian-motives-are-Artin-Tate}, there exists an integer $N\geq 1$
	such that $\Hetc^n(X_{\qbar},\bQ_p)$ is a  cyclotomic representation
	of $\Gal(\qbar/\bQ(\zeta_N))$. 
	The fact that the multiplicities are given by Hodge numbers as shown
	is because $\Hetc^n(X_{\qbar},\bQ_p)$ is potentially
	semi-stable.  See Kisin~\cite{Kisin:pst} (3.2).
	\marpar{more detail}
	
	Last, when $X$ is smooth and proper over $\bZ[1/M]$,
	the Galois representation $\Hetc^n(X_{\qbar},\bQ_p)$
	is unramified at primes not dividing $Mp$, by~\ref{pro:cohomology-is-constructible}.
	Then~\ref{thm:abelian-motives-are-Artin-Tate} implies that
	$N$ can be taken as asserted.
\end{proof}

\subsection{} \emph{Remark.} 
Note that the cohomology of the nodal curve of~\ref{subsec:singular-lines} when $q=-1$ is
non-pure mixed Tate, but the extension class is $q=-1\in\bQ^*$, which being torsion
vanishes when coefficients are taken in $\bQ_p$. Therefore the $\bQ_p$-cohomology is pure
mixed Tate, and there is no contradiction. On the other hand, the previous theorem would
be false with cohomology with coefficients in $\bZ_2$. It would be interesting to see
which other mixed Tate motives with torsion classes can be realized in $\Lambda$-algebraic
geometry.

\begin{corollary}\label{cor:zeta-poly-in-q}
	Let $N$ be a conductor for $X$.
	Then there is a finite set $\shfcond$ of prime numbers such that for any
	finite field $k$ whose cardinality $q$
	is relatively prime to every element of $\shfcond$  
	and satisfies $q\equiv 1\bmod N$, 
	the number of $k$-valued points of $X$ is $P(q)$.
	
	More precisely, 
	the set $\shfcond$ can be taken such that it contains only
	prime numbers $p$ with the property that for every prime $\ell\neq p$
	the sheaf $R^nf_!(\bQ_\ell)$ is not lisse at $p$, where 
	$f$ denotes the map 
		$$
		X\times\Spec\bZ[1/\ell]\to\Spec\bZ.
		$$
\end{corollary}

\begin{proof}	
	Fix a prime number $\ell\neq p$.
	Let $\shfcond$ denote the set of prime numbers at which $R^nf_!(\bQ_\ell)$ is 
	not lisse.	By~\ref{pro:cohomology-is-constructible}, 
	$\shfcond$ is finite.  
	
	Because of the restrictions on $q$, we have a factorization
		$$
		\Spec k\longlabelmap{a} \Spec\bZ[\zeta_N,T^{-1}] \longlabelmap{b} \Spec\bZ.
		$$
	Let $D_a$ denote a decomposition group in $\Gal(\bar{\bQ}/\bQ(\zeta_N))$ 
	at the point $a$,
	\marpar{Use cospecial-\\ization language?}
	and let $\bar{k}$ denote the corresponding algebraic closure of $k$.
	Then since $b^*R^nf_!(\bQ_{\ell})$ is lisse, the proper base change theorem
	implies 
		$$
		H^n(X_{\bar{k}},\bQ_{\ell}) \cong H^n(X_{\bar{\bQ}},\bQ_{\ell})
		$$
	as representations of $D_a$, where $D_a$ acts on the left side through
	$\Gal(\bar{k}/k)$.
	In particular, the trace of the geometric Frobenius element
	$F\in\Gal(\bar{k}/k)$ on $H^n(X_{\bar{k}},\bQ_{\ell})$ agrees with
	the trace of the inverse
	of an arithmetic Frobenius element $\Frob_a$ of $\Gal(\bar{\bQ}/\bQ(\zeta_N))$ on 
	$H^n(X_{\bar{\bQ}},\bQ_{\ell})$.

	Therefore by the Lefschetz fixed-point formula (Houzel, SGA 5, exp.\ XV \cite{SGA5}) 
	and~\ref{cor:lambda-varieties-are-Artin-Tate},
	the number of $\bF_q$-valued points of $X$ is
		\begin{align*}
			\sum_n (-1)^n \tr(F \mid \Hetc^n(X_{\bar{\bF}_q}, \bQ_\ell)) 
				&= \sum_n (-1)^n \tr(\Frob_a^{-1} \mid \Hetc^n(X_{\bar{\bQ}},\bQ_\ell)) \\
				&= \sum_n (-1)^n \tr(\Frob_a^{-1}\mid\bigoplus_m \bQ_\ell(-m)^{r_{m,n}}) \\
				&= \sum_n (-1)^n \sum_m r_{m,n} q^m = P(q).			
		\end{align*}
\end{proof}

\begin{corollary}\label{cor:smooth-proper-zeta-poly}
	Let $X$ be a smooth proper $\Lambda$-scheme over $\bZ[1/M]$.
\marpar{generalize to proper\\and locally acyclic?}
	Then there is an integer $N$ divisible only by the primes dividing $M$
	such that for all prime powers $q>1$ with $q\equiv 1\bmod N$,
	the number of $\bF_q$-valued points of $X$ is $P(q)$.
\end{corollary}

\begin{proof}
	The representation $\Hetc^n(X_{\qbar},\bQ_p)$ 
	of $\Gal(\bar\bQ/\bQ)$ is unramified away from $M$, and so
	there is an integer $N$ satisfying the property 
	of~\ref{cor:lambda-varieties-are-Artin-Tate}
	and such that it has the same prime divisors as $M$.

  	Now let $q>1$ be a prime power with $q\equiv 1\bmod N$.
	Let $p$ be the prime number dividing $q$, 
	let $\ell\neq p$ be another prime number.
	Observe that $p\nmid M\ell$.

	The map 
		$$
		f\:X\times\Spec\bZ[1/\ell]\to\Spec\bZ[1/M\ell]
		$$
	is smooth and proper, and so
	by Deligne (SGA 4 1/2, 
	[Arcata] V (3.1)~\cite{Deligne:SGA4.5}), the sheaf $R^nf_*(\bQ_\ell)$ is lisse.  
	In particular, it is lisse at $p$.
	Therefore by~\ref{cor:zeta-poly-in-q}, 
	there is a set of primes $T$ not containing $p$ such
	that the conclusion of~\ref{cor:zeta-poly-in-q} holds.
	Since $q\equiv 1\bmod N$, the number of $\bF_q$-valued points
	of $X$ is $P(q)$.
\end{proof}

\begin{corollary}\label{cor:}
	Let $X$ be a smooth proper scheme over $\bF_1$.
	Then for all prime powers $q>1$, the number of $\bF_q$-valued points
	of $X$ is $P(q)$.
\end{corollary}

\begin{proof}
	This is~\ref{cor:smooth-proper-zeta-poly} in the case $M=1$.
\end{proof}

\subsection{} \emph{Remark.}
\label{subsec:F1-points-on-Chebychev}
The number of complemented $\bF_1$-valued points of $X$ equals $P(1)$
when $X=\bA^n$ or $X=\bP^n$, with their toric $\Lambda$-structures.
On the other hand, if $X$ is the Chebychev line~(\ref{subsec:chebychev-line}), 
it has no complemented $\bF_1$-valued points, but $P(1)=1$.

\subsection{} \emph{Remark.}
Some conditions on $q$ of the kind in~\ref{cor:zeta-poly-in-q} are necessary.
For example, let 
$X=\Spec\bZ[\zeta_N,1/NM]$.  For primes $p$ dividing $NM$, let $\psi_p=\id$;
and for all other $p$, let $\psi_p$ be the unique Frobenius lift.

Then $r_{0,0}=\phi(N)$ and all other $r_{m,n}$ are 
zero.  Therefore $P(t)=\phi(N)$.  On the other hand,
there are $\phi(N)$ points in $X(\bF_q)$
if and only if both $NM\in\bF_q^*$ and $q\equiv 1\bmod N$.
Furthermore the sheaf $R^nf_!(\bQ_{\ell})$ is lisse exactly at the prime numbers
that do not divide $NM$.

\begin{proposition}\label{pro:weight-constraints}
\label{lem:lambda-spaces-have-Fq-points}
	Let $X$ be a nonempty separated $\Lambda$-scheme of finite type over $\bZ$.
	Let $P(t)$ and $N$ be as in~\ref{thm:lambda-varieties-are-abelian}.
	Then for every prime $p\gg 0$ and every integer $r\geq 1$ with
	$p^r\equiv 1 \bmod N$, there is an $\bF_{q}$-valued point of $X$.
\end{proposition}
\begin{proof}
	By~\ref{cor:zeta-poly-in-q}, for every prime $p$ not in set $T$ supplied and for
	every integer $r$ with $p^r\equiv 1\bmod N$, the number of $\bF_{p^r}$-valued
	points of $X$ is $P(p^r)$.

	On the other hand, it follows from general facts about
	Hodge numbers of varieties
	proved by Deligne~\cite{Deligne:HodgeIII}, (8.2.4),
	that $P(t)\to\infty$ as $t\to\infty$.
	Indeed, since $X_{\bC}$ is nonempty, its dimension $d$ is non-negative,
	the degree of $P(t)$ is $2d$, and the coefficient of $t^{2d}$ is the
	number of connected components of $X_{\bC}$, which is positive.

 	Thus for sufficiently large $p$ and all $r$ as above, we see
	that the number of $\bF_{p^r}$-valued points
	is $P(p^r)$ and that $P(p^r)$ is positive.
\end{proof}

\begin{theorem}\label{thm:anothertry}
	Let $X$ be a nonempty separated $\Lambda$-scheme of finite type over $\bZ$.
	Then there is a nonempty closed $\Lambda$-subscheme $Z$ of $X$ which is
	\'etale over $\bZ$.
\end{theorem}
\begin{proof}
	Let us first reduce to the case where
	$X$ is reduced and quasi-finite over $\bZ$.
	Define a sequence $X_0\supseteq X_1\supseteq\cdots$ 
	of nonempty closed sub-$\Lambda$-schemes of $X$ recursively as follows:
	Let $X_0=X$.  For $n\geq 0$, assume $X_n$ has already been defined.
	Then let $p=p_{n+1}$ be a prime number distinct from $p_1,\dots,p_n$ 
	such that $X_n$ has an $\bF_{p}$-valued point $x$.
	This exists by~\ref{lem:lambda-spaces-have-Fq-points}.
	(For definiteness, we can take $p$ to 
	be the smallest such prime, say.)  Let $X_{n+1}$ be the fixed locus of $X_n$ under 
	$\psi_{p}$, which is to say the equalizer in the category of $\Lambda$-spaces
	of $\psi_{p}\:X_n\to X_n$ and the identity map.  Because $X_n$ is separated,
	$X_{n+1}$ is a closed sub-$\Lambda$-scheme of $X_n$.  
	Because $x$ is an
	$\bF_{p}$-valued point of $X_n$, it is fixed by the Frobenius map
	on the special fiber of $X_n$ over $p$.  Therefore it is fixed by $\psi_p$,
	and so $x$ is also a point of $X_{n+1}$.
	In particular, $X_{n+1}$ is nonempty.

 	Let $Z= \bigcap_{n\geq 0} X_n$.  
	Because $X$ is of finite type over $\bZ$, it
	is noetherian.  Therefore there is an integer $n\geq 0$ such that $Z=X_n$.
	It follows that $Z$ is nonempty and, by~\ref{pro:finiteness-of-periodic-spaces}, 
	it is also affine and quasi-finite
	over $\bZ$. Therefore we can assume $X=Z=\Spec B$, 
	where $B$ is quasi-finite over $\bZ$.
	Since the reduced subscheme of any nonzero $\Lambda$-ring of
	finite type over $\bZ$ is the same, we can also assume $B$ is reduced.
	
	Now let us show that $B$ has a quotient $\Lambda$-ring which is \'etale over $\bZ$.
	Suppose $B$ is not \'etale over $\bZ$ at some prime $p$.
	For each integer $m\geq 1$, let $I_m$ denote the kernel of $\psi_m\:B\to B$.
	The $\Lambda$-ideals $I_m$ are ordered by divisibility on $m$,
	and the ordering is cofinal.
	Let $I$ denote the $\Lambda$-ideal $\cup_m I_m$, and
	let $C$ denote the quotient of $B/I$ by the ideal of torsion elements.
	Then $C$ is a $\Lambda$-ring quotient of $C$ and, hence, of $B$.
	\marpartd{ref}
	Note that $1\not\in I$, so $B/I$ is a nonzero $\Lambda$-ring.
	But	$1$ is not a torsion element in any nonzero $\Lambda$-ring.  	
	Therefore $C$ is nonzero and is flat, quasi-finite, and of finite type over $\bZ$.
	Let us finally show that $C$ is actually \'etale over $\bZ$.
	
	For each integer $m\geq 1$,
	the endomorphism $\psi_m$ of $C$ is injective.  Indeed, if $b$ is a lift
	to $B$ of any element of the kernel, then $n\psi_m(b)=0$ for some integer $n\geq 1$.
	Therefore $\psi_m(nb)=0$ and hence $nb=0$ and hence $b$ is torsion.  Therefore
	the image of $b$ in $C$ is $0$.
	
	Now let $p$ be a prime.
	Since $\psi_p$ is an injective endomorphism of $C$, it induces an injective
	endomorphism of $\bQ \tn_{\bZ} C$.  Since $\bQ\tn_{\bZ}C$ is finite over $\bQ$,
	this endomorphism
	is in fact an automorphism of finite order.  Since $C$ is flat over $\bZ$,
	we have $C\subseteq \bQ\tn_{\bZ}C$, and so $\psi_p$ is an automorphism of $C$.
	Therefore the Frobenius endomorphism of $C/pC$ is an automorphism.
	Thus $C/pC$ is reduced, and so $C$ is \'etale at $p$.  
\end{proof}

\begin{corollary}\label{cor:F1-points-exist}
	Let $X$ be a nonempty proper scheme over $\bF_1$.  Then $X(\bF_1)\neq\emptyset$.
\end{corollary}
\begin{proof}
	The scheme $Z$ supplied by~\ref{thm:anothertry} 
	is proper and \'etale over $\bZ$.  Therefore each of its finitely many
	connected components must be $\Spec\bZ$, by Minkowski's theorem.
 	But the only $\Lambda$-structure on such a space
	is the disjoint-union $\Lambda$-structure.  Indeed, it is flat over $\bZ$, so it suffices to
	show that every Frobenius lift $\psi_p$ is the identity.  But the Frobenius map on each
	special fiber is the identity.  Since $Z$ is a disjoint union of
	copies of $\Spec\bZ$, each $\psi_p$ must be the identity.
	
	Therefore $Z$, as a nonempty disjoint union of copies of $\Spec\bF_1$, has an
	$\bF_1$-valued point, and hence so does $X$.
\end{proof}

\begin{corollary}\label{cor:F1-points-exist2}
	Let $U$ a nonempty open $\Lambda$-subscheme of a
	$\Lambda$-scheme $X$ which is proper over $\bZ$.  Then $U$ has a $\bQ$-valued
	$\Lambda$-point.
\end{corollary}
\begin{proof}
	By the theorem above, there is a nonempty closed $\Lambda$-subscheme $Z$ of $U$
	which is \'etale over $\bZ$.
	Let $Y$ denote the closure of $Z$ in $X$ with the reduced subscheme structure.
	Then $Y$ is a closed $\Lambda$-subscheme of $X$.  (Basic property of $\Lambda$-ideals.)
	Because $Y$ is reduced, it is flat.  Because it is the closure of $Z$ it is
	generically finite over $\bZ$.
	A closed subscheme of $X$, it is therefore finite over $\bZ$.
	By~\ref{cor:F1-points-exist}, it has an $\bZ$-valued $\Lambda$-point,
	and hence a $\bQ$-valued $\Lambda$-point.
	Since we have 
		$$
		\Spec\bQ\times_{\Spec\bZ} Y = \Spec\bQ\times_{\Spec\bZ}Z 
			\subseteq \Spec\bQ \times_{\Spec\bZ} U,
		$$
	we see that $U$ has a $\bQ$-valued $\Lambda$-point.
\end{proof}

\subsection{} \emph{Remark.}
The condition that such a compactification $X$ exists 
cannot be dropped.  For example, for any integer $n>0$,
we can make $\bZ[\zeta_n,1/n]$ a $\Lambda$-ring by taking $\psi_p$ to be
anything for $p\mid n$, and to be the unique
choice $\zeta_n\mapsto\zeta_n^p$ for $p\nmid n$.  But
this ring has no maps to $\bQ$.

\subsection{}
	For any $\bF_1$-valued point $x$ of $X$, 
	let $Z_x$ denote the closure of the pre-images of $x$ under the maps
	$\psi_p$, viewed as a reduced closed subscheme.  
	Then $Z_x$ is a complemented closed $\Lambda$-subspace of $X$.
	Therefore the assignment $x\mapsto Z_x$ defines a function 
		$$
		\{\text{$\bF_1$-points of $X$}\}
			\longmap \{\text{complemented reduced closed $\Lambda$-subspaces of $X$}\}.
		$$

\begin{corollary}\label{cor:generic-f1-point}
	Let $X$ be a proper $\Lambda$-scheme over $\bZ$.  
	Assume that $X$ is irreducible as a scheme.  Then there is a
	$\bZ$-valued $\Lambda$-point $x$ such that $Z_x=X$.
\end{corollary}
\begin{proof}
	Since the set $X(\bF_1)$ is finite and complementary closed $\Lambda$-subschemes
	are stable under finite union, the union 
		$$
		Z=\bigcup_{x\in X(\bF_1)} Z_x,
		$$
	with its reduced scheme structure is a complementary closed $\Lambda$-subscheme.
	By construction, there are no $\bQ$-valued $\Lambda$-points of $X-Z$.
	Therefore by~\ref{cor:F1-points-exist2}, this is only possible if $X=Z$.
	Finally since $Z$ is irreducible, we must have $Z_x=X$ for some $x\in X(\bF_1)$.
\end{proof}

%
%
%
%

\section{Variants}
\label{sec:variations}

One of the motivations behind work on the field with one element has been
to imitate over number fields the theory of 
function fields over a finite field $k$, where we can work over the absolute point 
$\Spec k$.  But $\Lambda$-algebraic geometry works perfectly well over 
function fields, too.  So we can compare $k$-algebras to the function-field 
analogues of $\bF_1$-algebras.

\subsection{} \emph{$\Lambda_{S,\ptst}$-spaces.}
Let $S$ be a scheme of finite type over $\bZ$, and let $\ptst$ be a set of
regular closed points of codimension $1$. For each point $s\in\ptst$, let $q_s$ denote the
cardinality of $s$. Let $X$ be a flat algebraic space over $S$. We can then define a
$\Lambda_{S,\ptst}$-action on $X$ just as we define $\Lambda$-actions, but now we use
commuting endomorphisms $\psi_s\:X\to X$, one for each point $s\in\ptst$, such that $\psi_s$
agrees with the $q_s$-power Frobenius operator on the fiber $X_s$. (See~\cite{Borger:SLAG}.)
If $\ptst$ is the set of all regular closed points of codimension $1$, then we write
$\Lambda_S=\Lambda_{S,\ptst}$.  For example, if $S=\Spec \bZ$, then a $\Lambda_S$-space
is just a $\Lambda$-space.

\newcommand{\tmapa}{a}
\newcommand{\tmapb}{b}
The general theory works just as well when $S$ is arbitrary.  Thus we get a topos
$\Space_{\bF_1^{S,\ptst}}$, whose objects we call spaces over the generalized
field with one element $\bF_1^{S,\ptst}$.  If $S'$ and $\ptst'$ are another instance
of this data, and $\tmapa\:S\to S'$ is a map such that $a(\ptst)\subseteq \ptst'$,
we have a diagram of toposes
\begin{equation} \label{diag:change-of-S-and-E}
		\entrymodifiers={+!!<0pt,\fontdimen22\textfont2>}
		\def\objectstyle{\displaystyle}		
	\xymatrix@R=30pt@C=40pt{
	\Space_{S} \ar^{\ltm}[r] \ar^{\tmapa}[d] 
		& \Space_{\bF_1^{S,\ptst}} \ar^{\tmapb}[d] \\
	\Space_{S'} \ar^{\ltm'}[r]
		& \Space_{\bF_1^{S',\ptst'}}
	}	
\end{equation}
rendered commutative by a certain invertible 2-morphism. \marpar{precise refs}%
Here $\tmapb$ is as in~\cite{Borger:SLAG}.

In particular, $\Space_{\bF_1}$ is the deepest.

\subsection{} \emph{Function fields.}
Let $q$ denote the cardinality of $k$, and let $S$ be a smooth geometrically connected
curve over $k$. We will now construct a factorization of topos maps
	\begin{equation} \label{eq:diag-func-field-comm-topos}
	\xymatrix{
	\Space_S\ar^-{\ltm}[rr]\ar^-{\stmap}[dr]
		&	& \Space_{\bF^S_1}\ar@{-->}_-{\fftm}[dl] \\
		& \Space_{\Spec k}.
	}
	\end{equation}

Let us first define $\ffaus(T)$ for affine (or algebraic) $T$. As a space, set 
	$$
	\ffaus(T)=S\times_k T.
	$$
Since $S\times_k T$ is flat over $S$, giving a $\Lambda_S$-action on $S\times_k T$, is the same as
giving a commuting family of Frobenius lifts. For any maximal ideal $\m$ of $\sO_S$, define
$\psi_\m\colon S\times_k T\to S\times_k T$ by $\psi_{\m} = \id_S\times\Fr^{\m}_T$, where
$\Fr^{\m}_T$ denotes the endomorphism of $T$ which on $\sO_T$ acts as $x\mapsto x^{q_\m}$, where
$q_{\m}$ denotes the cardinality of the residue field $\sO_S/\m$. It is clear the $\psi_\m$ commute
with each other and lift the appropriate Frobenius maps.

The functor $\aff_S\to\Space_{\bF^S_1}$ just defined preserves covering families. Indeed, a map
$U\to V$ in $\Space_{\bF^S_1}$ is an epimorphism if and only if the induced map $\ltm^*U\to\ltm^*V$
is. But $\ltm^*\ffaus$ preserves covering families. Therefore $\ffaus$ does.

For a similar reason, $\ffaus$ sends products to products. Therefore it extends uniquely to a topos
map $f\:\Space_{\bF^S_1}\to\Space_{\Spec k}$, which yields the commutative
diagram~(\ref{eq:diag-func-field-comm-topos}).

In fact, $f$ is essential, meaning that $\ffaus$ has a left adjoint, denoted $\ffalsh$.
We let $\ffalsh(U)$ be the colimit of the coequalizers of the diagrams
	\[
	\displaylabelrightarrows{U}{\Fr^{\m}_U}{\psi_{\m}}{U}
	\]
over all $\m$.  In other words, $\ffalsh(U)$ is the largest quotient $U'$ of $U$ on which
each $\psi_\m$ acts as $\Fr^{\m}_{U'}$.   This is clearly the left adjoint of $\ffaus$.

\subsection{} {\em $\fftm$ is not an isomorphism of toposes.}
For example, let $T$ be an affine space over $S$, the unit of the adjunction
$\fftm_! \dashv \fftm^*$ at $\ltm_!(T)$ is
	\begin{equation} \label{eq:frob-component-map}
 	\ltm_!(T) \longmap \fftm^*\fftm_!\ltm_!(T) = \fftm^*\stmap_!(T)
	\end{equation}
which, by definition, is a map
	\begin{equation} \label{eq:can-of-worms}
		W_S(T) \longmap S\times_k T,
	\end{equation}
where $W_S:=W_{S,\ptst}$ is the $E$-typical Witt vector functor over $S$.
(See~\cite{Borger:SLAG}.)\marpar{ref}
The composition
	\[
	\coprod_{\nset}T \longlabelmap{\gh{}} W_S(T) \longmap S\times_k T.
	\]
with the ghost map $\gh{}$ is the map that, on the component $n\in\nset$, is simply
	$$
	\entrymodifiers={+!!<0pt,\fontdimen22\textfont2>}
	\def\objectstyle{\displaystyle}
	\xymatrix@C=40pt{
	T \ar^-{(\pr,\Fr^{n}_T)}[r] & S\times_k T
	}
	$$
where $\pr$ denotes the structure map $T\to S$, and $\Fr^n_T$ is the Frobenius
map defined on functions
by $x\mapsto x^{q^{\deg(n)}}$, where $\deg(n)$ denotes $\sum_\m n_\m[\sO_S/\m:k]$,
the degree of the effective divisor corresponding to $n$.

In particular, if $T=S$, then the image of this map is the union in $S\times_k S$ of the graphs
of all powers of the Frobenius map on $S$. These components are not disjoint. For instance, let
$x$ and $y$ be two distinct closed points of $S$ with the same residue field. Then the the
components of $W_S(S)$ of indices $x$ and $y$ are distinct but have the same image in $S\times_k
S$. So~(\ref{eq:frob-component-map}) is not a monomorphism. Therefore $\fftm_!$ is not faithful,
and hence $f$ is not an isomorphism of toposes. One can also show
that~(\ref{eq:frob-component-map}) is not an epimorphism.

But the image of~(\ref{eq:can-of-worms}), or equivalently~(\ref{eq:frob-component-map}),
does see much of the geometry of $S\times_k T$.
Another way of expressing the point of this paper is that it is reasonable
to think of it as being almost an isomorphism.  For example, we have the following result.

\begin{proposition}\label{pro:W-is-dense-in-product}
	The map
		$$
			W_S(T) \longmap S\times_k T
		$$
	of~(\ref{eq:can-of-worms}) has Zariski dense image.
\end{proposition}
\begin{proof}
	Let $\m$ be a maximal ideal of $\sO_S$ with residue cardinality $r$.
	It suffices to show the image of the composition
		$$
		\bN \times T \longmap \nset \times T \longmap W_S(T) \longmap S\times_k T
		$$
	is dense.  The map satisfies $(n,t)\mapsto \big(h(t),\Fr_{r^n}(t)\big)$.
	
	We may assume $S$ and $T$ are affine.  Indeed,
	it suffices to show density locally.  Therefore we can replace $S$ with an affine
	open subscheme $S'$ and replace $T$ with $T'=S'\times_S T$.  We can then replace $T'$
	with an affine scheme $T''$ mapping to $T'$ by an \'etale map.

	Let us then write $S=\Spec R$ and $T=\Spec B$, where $R$ is a Dedekind
	domain and $B$ is an $R$-algebra.  
	In terms of rings, the map in question is
	\begin{equation} \label{eq:internal-density-map}
		R\tn_k B \longmap B\times B\times\dots
	\end{equation}
	and is defined by $a\tn b \mapsto \ang{ab,ab^r,ab^{r^2},\dots}$.
	We need to show that any element of its kernel is nilpotent.
	
	So let $\sum_{j=1}^d a_j\tn b_j$ be an element of its kernel.  Assume without
	loss of generality that the elements $a_j$ are linearly independent over $k$.
	Then for every $m\in \bN$, we have
	$\sum_j a_jb_j^{r^m}=0$.  Applying various powers of $\Fr_{r}$, we have the system
	of equations
		$$
		\sum_{j=1}^d a_j^{r^i}b_j^{r^d} = 0
		$$
	where $i=0,\dots,d-1$.  This can be expressed as the matrix equation
	$$
	\big(a_j^{r^i}\big)_{ij}\cdot\big(b_j^{r^d}\big)_j=0.
	$$  
	Since the family $a_1,\dots,a_d$ is linearly independent, the matrix $\big(a_j^{r^i}\big)$
	has nonzero determinant.  (This is the
	Moore matrix from the theory of function fields.
	The vanishing of its determinant is
	equivalent to the linear dependence of the $a_j$.  The proof is by a degree
	argument, just as for the familiar, analogous result about Vandermonde matrices.)
	Therefore we have $b_j^{r^d}=0$ for all $j$.  It follows that the element
		$$
		\Big(\sum_{j=1}^d a_j\tn b_j\Big)^{r^d}
		$$ 
	is zero.  So every element of the kernel of~(\ref{eq:internal-density-map})
	is nilpotent.
\end{proof}

\begin{corollary}\label{cor:separated-univ-prop-for-function-fields}
	Let $X$ and $Y$ be separated reduced algebraic spaces over $k$.
	Then any map $\ffalsh\ffaus X\to Y$ factors uniquely through
	the map $\ffalsh\ffaus X\to X$.
\end{corollary}
\begin{proof}
 	Consider the composite map
		$$
		S\times X \longmap \ffalsh \ffaus X \longmap Y,
		$$
	and let $\Gamma=(S\times X)\times_Y (S\times X)$ denote the induced equivalence
	relation on $S\times X$.  Recall that $\ffalsh\ffaus X$ is defined to be the quotient
	of $S\times X$ by the equivalence relation $\Gamma'$ generated by the image of
	$W(S\times X)$ in $S\times S \times X$
\marpartd{Clean up.\\Where is reduced\\used?}
	under the map~(\ref{eq:can-of-worms}) when $T=S\times_k X$.  Therefore
	$\Gamma$ contains $\Gamma'$.  On the other hand, by~\ref{pro:W-is-dense-in-product},
	$\Gamma'$ is dense in $\Gamma''= S\times S\times X = (S\times X)\times_X(S\times X)$.
	Therefore $\Gamma$ agrees with $\Gamma''$, and so the map $S\times X\to Y$
	factors uniquely through the quotient of $S\times X$ by $\Gamma''$, which is just $X$.
\end{proof}

\begin{corollary}\label{cor:f^*-faithful-and-almost-full}
	The functor $\ffaus\:\Space_{k}\to\Space_{\bF^S_1}$ is faithful,
	and it embeds the full subcategory of separated reduced algebraic spaces 
	over $k$ fully faithfully in the category $\Space_{\bF^S_1}$.
\end{corollary}

Some restriction on nilpotent elements is necessary. For example, suppose $S=\Spec R$. Let $A$
be a non-reduced $k$-algebra. Choose a square-zero element $a\in A$ and an element $r$ in $R$
but not in $k$. Then the element $r\tn a$ of $R\tn_k A$ does not lie in the subring $A$. On the
other hand, we have $(r\tn a)^{q_\m} = 0 = r\tn a^{q_\m}=\psi_\m(r\tn a)$ for any maximal ideal
$\m$ of $R$. Therefore the map $\ffalsh\ffaus\Spec A\to\Spec A$ is not an isomorphism, because
$r\tn a$ is a function on the affinization of $\ffalsh\ffaus\Spec A$ that does not come from
$A$.

\subsection{} {\em Analogy.}
It is rare for the pull-back functor for a map of spaces to be fully faithful.  
So let us consider a similar, but more familiar situation where this happens.
Let $S$ be a complex algebraic space
and let $\Gamma$ be an equivalence relation on $S$ which is Zariski dense in $S\times S$.
For instance, $\Gamma$ could be given by the action of a discrete group with a dense orbit 
or, if $S$ is connected, by the formal neighborhood of the diagonal.  Then algebraic 
spaces (perhaps under some mild conditions) form a full subcategory of the category of
$\Gamma$-equivariant spaces over $S$.  The condition for a $\Gamma$-equivariant
algebraic space $X$ over $S$ to descend to the point is then a property on $X$, rather
than a structure.  We might then say the $\Gamma$-action is uniform, or constant.

Therefore, following the previous corollary, it is natural to interpret a descent 
datum on a reduced algebraic space $X$ over $S$ to the point $\Spec k$ as being a 
descent datum to $\bF^S_1$ with a similar algebraic uniformity property.
So, objects of $\bF^S_1$ are generalized---but not weakened---versions of separated
reduced algebraic spaces over the point $\Spec k$.  Of course
this makes essential use of equal characteristic.  The 
corresponding interpretation of $\Lambda$-spaces in the usual sense, over $\bZ$, would 
be that while it is possible to say what it means to descend
an algebraic space to $\bF_1$---that is, to give it a $\Lambda$-action---we
do not know if there is a uniformity property, which
is what we would need to create a true arithmetic analogue of the base point $\Spec k$.

(Buium has come to similar ideas independently. He proposed in conversation that
it might be reasonable to consider a $\Lambda$-structure on a scheme as being an
isotrivialization relative to $\bF_1$.)

\subsection{} {\em Drinfeld modules.}
Using the theory of Drinfeld modules, we can give
examples of objects of $\Space_{\bF^S_1}$ that do not descend to $\Space_{k}$.

Let $C$ be a connected smooth projective curve over $\bF_p$, let $\infty\in C$ be a closed
point, and let $A=\Gamma(C-\{\infty\},\sO_C)$.  Let $S$ be a smooth $\bF_p$-curve 
over which there is a Drinfeld $A$-module
	$$
	\varphi\:A\to \End_S(\Ga)
	$$ 
of rank $1$ and of generic characteristic.  
(See Laumon~\cite{Laumon:Drinfeld-book-I} section (1.2), say.)

Then for any closed point $s$ of $S$,
the fiber of $\varphi$ over $s$ gives a Drinfeld
$R$-module $\varphi_{s}$ over $s$ of rank $1$.  Because of the assumption 
that $\varphi$ has generic characteristic,
the characteristic of $\varphi_s$ is $s$.  A basic result of Drinfeld's
(\cite{Laumon:Drinfeld-book-I} (2.2.2)(ii)) 
then implies that there is a unique element $\Pi_s\in A$ 
such that $\varphi_s(\Pi_s)$ is the $q_s$-th power Frobenius endomorphism of $\Ga$ over $s$,
where $q_s$ denotes the residue cardinality of $s$.
If we set $\psi_s=\varphi(\Pi_s)$, then the various $\psi_s$ are commuting
endomorphisms of $\Ga$ over $S$, each agreeing with the $q_s$-power Frobenius map
on the fiber over $s$.  This gives a $\Lambda_S$-structure on $\Ga$ (which also
respects the group structure).

For example, the Carlitz module is defined when $A=\bF_p[t]$ and $S=\Spec A$ by 
$\rho(t)=t+\tau$, where $\tau$ is the Frobenius map of $\Ga$.
Then for each maximal ideal $\m$ of $k[t]$, the operator $\psi_\m$ is
$\rho(f(t))$, where $f(t)$ denotes the monic generator of $\m$. 

Observe that none of these ``Drinfeld $\Lambda_S$-structures'' on $\bA^1_S$ 
descends from $\Space_{\bF^S_1}$ to $\Space_k$.
Indeed, for every object in the image of $\ffaus$,
the operators $\psi_s$ act as zero on the conormal sheaf of the identity
section $S\subset\Ga$.  But $\varphi$ was assumed to have generic
characteristic.  Therefore every $\Pi_s$ acts faithfully on the conormal sheaf,
and hence so does every $\psi_s$.

Another important use of the construction $S\times X$ is in the study of shtukas, 
also due to \marpartd{delete?}
Drinfeld.  Indeed, it is possible to mimic this using     
$W_S(X)$ instead of $S\times X$.  And this concept can be translated to number fields.
This paper is not, however, 
the place to discuss this in any detail.

\subsection{} \emph{Complex multiplication by number fields.}
Let $R$ be a Dedekind domain whose field of fractions is a number field.
Assume that there is an abelian scheme $X$ over $R$ of dimension $d$
having the property that $\bQ\tn\End_R(X)$ contains a field $F$ of degree $2d$ over $\bQ$.

\hide{
PROBABLY $F$ HAS TO EQUAL $\bQ\tn\End_R(X)$. \marpartd{work} 
}

As above, $X$ has a natural $\Lambda_R$-structure.  For each maximal ideal $\m$ of $R$,
there is a unique element $\pi_\m\in F\cap\End_R(X)$ such that $\pi_\m$ induces the 
$q_\m$-th power Frobenius map on the fiber of $X$ over $\m$.
(See Serre--Tate~\cite{Serre-Tate}.)  Because each $\pi_\m$
lies in $F$, they all commute.  Therefore putting $\psi_\m=\pi_\m$ is 
a $\Lambda_R$-structure on $X$.

Observe that we can modify $X$ to make $\Lambda_R$-varieties that are not 
CM varieties in the usual sense.  For instance, let $G$ is a finite subgroup of
$\Aut_R(X)$.  Because every automorphism commutes with the complex multiplications,
\marpartd{true?}
$G$ acts $\Lambda_R$-equivariantly on $X$.  Therefore the quotient $X/G$ is
also a $\Lambda_R$-space.  (Because $G$ is finite, the quotient is an algebraic
space, by Artin's theorem.) \marpartd{need ref}
For instance, if $X$ is an elliptic curve
and $G=\Aut(X)$, then $X/G$ is a projective line.  

In particular, it seems likely that
explicit class field theory for imaginary quadratic fields could be expressed
in terms of $\Lambda_R$-structures on $\bP^1$, just like in the case
of $\bQ$ and function fields.

Suppose instead that $R$ is a Dedekind domain whose field of fractions is a real quadratic number
field. In light of Ritt's work~\cite{Ritt:Permutable}\cite{Ritt:Permutable-errata}, it seems
unlikely that there are $\Lambda_R$-actions on $\bP^1$ which do not come from
$\Lambda_{\bZ}$-actions. It might, however, be possible to find such $\Lambda_R$-actions on
surfaces or higher-dimensional varieties, and any example would without a doubt lead to another
example of an explicit class field theory. Of course, it would only be interesting if it
could see more than the maximal cyclotomic extension of $R$. On the other hand,
it would also be interesting to prove that no such examples exist.  

Here are some precise questions.
Is there an algebraic $\Lambda_R$-space $X$ of finite type over $R$ with the property that
for any abelian \'etale $R$-algebra $R'$, the 
$\Lambda_R$-space $\Spec R'\times_{\Spec R} X$
is not isomorphic as a $\Lambda_R$-space to one of the form
$\Spec R'\times_{\Spec \bZ} Y$, for
any algebraic $\Lambda_\bZ$-space $Y$?  (\emph{Abelian} here
means that $\bQ\tn_\bZ R'$ is a product of abelian extensions of $\bQ\tn_\bZ R$.)
Are there any algebraic $\Lambda_{R}$-spaces $X$ of finite type over $R$ whose
generic fiber is geometrically connected and which do not descend to $\Lambda_{\bZ}$-spaces?
I do not even know if many particular varieties can be ruled out:
are there any $\Lambda_R$-structures on $X=\bP^2_{R}$ with this property?

\subsection{} {\em Local number fields.}
\label{subsec:local-absolute-point}
Let $S=\Spec\bZ$, let $p$ be a prime number, and let $\ptst=\{p\}$. Let us write
$\Lambda_{S,\ptst}=\Lambda_p$. The corresponding Witt functor $\wnus{n}=\wnus{S,\ptst,n}$
is, up to re-indexing, the $p$-typical Witt vector functor defined by
Witt~\cite{Witt:Vectors} in 1936. \marpar{1937?}
Of course, $\Lambda_p$-rings and $p$-typical Witt
vectors are now ubiquitous in work on $p$-adic cohomology. (See~\cite{deJong:ICM} or 
\cite{Illusie:dRW-1016}, say.)

Now let $A$ be a complete discrete valuation ring with perfect residue field $k$ of
characteristic $p$. If $A$ is of equal characteristic, then the map $A\to k$ has a unique
section. Now suppose $A$ has mixed characteristic. Of course $A\to k$ cannot have a
section defined over $\bZ$, but remarkably, it does have a unique section defined over
$\bF_1^{S,\ptst}$, the $p$-typical field with one element, and hence over $\bF_1$. By
definition, this means that the map $W(A)\to W(k)$ of $\Lambda_p$-rings has a unique
section. Indeed, because $k$ is perfect, there is a unique ring map $W(k)\to A$ compatible
with the projections to $k$. (See~\cite{Serre:CL}, say.) By adjointness, this then lifts
to a unique map $W(k)\to W(A)$ of $\Lambda_p$-rings, which is easily seen to be a section
of the map in question. \marpar{Lubin--Tate\\groups?}


\bibliography{references}

\begin{thebibliography}{10}

\bibitem{SGA4.1}
{\em Th\'eorie des topos et cohomologie \'etale des sch\'emas. {T}ome 1:
  {T}h\'eorie des topos}.
\newblock Springer-Verlag, Berlin, 1972.
\newblock S\'eminaire de G\'eom\'etrie Alg\'ebrique du Bois-Marie 1963--1964
  (SGA 4), Dirig\'e par M. Artin, A. Grothendieck, et J. L. Verdier. Avec la
  collaboration de N. Bourbaki, P. Deligne et B. Saint-Donat, Lecture Notes in
  Mathematics, Vol. 269.

\bibitem{SGA5}
{\em Cohomologie {$l$}-adique et fonctions {$L$}}.
\newblock Lecture Notes in Mathematics, Vol. 589. Springer-Verlag, Berlin,
  1977.
\newblock S{\'e}minaire de G{\'e}ometrie Alg{\'e}brique du Bois-Marie
  1965--1966 (SGA 5), Edit{\'e} par Luc Illusie.

\bibitem{Baez:TWF184}
John Baez.
\newblock This week's finds in mathematical physics (week 184).
\newblock http://www.math.ucr.edu/home/baez/week184.html.

\bibitem{Borger-deSmit:integral-lambda-models}
James Borger and Bart de~Smit.
\newblock Galois theory and integral models of ${\Lambda}$-rings.
\newblock {\em Bulletin of the London Mathematical Society}.
\newblock To appear.

\bibitem{Borger-Wieland:PA}
James Borger and Ben Wieland.
\newblock Plethystic algebra.
\newblock {\em Adv. Math.}, 194(2):246--283, 2005.

\bibitem{Borger:BGWV}
James~M. Borger.
\newblock The basic geometry of {W}itt vectors.
\newblock arXiv:0801.1691v2.

\bibitem{Borger:SLAG}
James~M. Borger.
\newblock Sheaves in {$\Lambda$}-algebraic geometry.
\newblock To appear.

\bibitem{Buium:differential-subgroups}
Alexandru Buium.
\newblock Differential subgroups of simple algebraic groups over {$p$}-adic
  fields.
\newblock {\em Amer. J. Math.}, 120(6):1277--1287, 1998.

\bibitem{Buium:Arithmetic-diff-equ}
Alexandru Buium.
\newblock {\em Arithmetic differential equations}, volume 118 of {\em
  Mathematical Surveys and Monographs}.
\newblock American Mathematical Society, Providence, RI, 2005.

\bibitem{Buium:geometry-of-fermat-adeles}
Alexandru Buium.
\newblock Geometry of {F}ermat adeles.
\newblock {\em Trans. Amer. Math. Soc.}, 357(3):901--964 (electronic), 2005.

\bibitem{Buium-Simanca:Arithmetic-Laplacians}
Alexandru Buium and Santiago~R. Simanca.
\newblock Arithmetic {L}aplacians.
\newblock {\em Adv. Math.}, 220(1):246--277, 2009.

\bibitem{Clauwens:Line}
F.~J. B.~J. Clauwens.
\newblock Commuting polynomials and {$\lambda$}-ring structures on {${\bf
  Z}[x]$}.
\newblock {\em J. Pure Appl. Algebra}, 95(3):261--269, 1994.

\bibitem{Connes-Consani:notion-of-geometry-over-f1}
Alain Connes and Caterina Consani.
\newblock On the notion of geometry over {${\bF_1}$}.
\newblock arXiv:math/0809.2926, version dated 2008/09/17.

\bibitem{Connes-Consani-Marcoli:fun-with-f1}
Alain Connes, Caterina Consani, and Matilde Marcoli.
\newblock Fun with {${\bF_1}$}.
\newblock arXiv:math/0806.2401, version dated 2008/06/14.

\bibitem{Conrad:Nagata-compact}
Brian Conrad.
\newblock Deligne's notes on {N}agata compactifications.
\newblock {\em J. Ramanujan Math. Soc.}, 22(3):205--257, 2007.

\bibitem{Davydov:periodic-lambda-rings}
A.~A. Davydov.
\newblock Periodic {$\lambda$}-rings and periods of finite groups.
\newblock {\em Mat. Sb.}, 188(8):75--82, 1997.

\bibitem{deJong:ICM}
A.~J. de~Jong.
\newblock Barsotti-{T}ate groups and crystals.
\newblock In {\em Proceedings of the International Congress of Mathematicians,
  Vol. II (Berlin, 1998)}, number Extra Vol. II, pages 259--265 (electronic),
  1998.

\bibitem{Deitmar:F1-and-toric}
Anton Deitmar.
\newblock ${\bF}_1$-schemes and toric varieties.
\newblock arXiv:0608.5179, version dated 2007/11/29.

\bibitem{Deitmar:F1-schemes}
Anton Deitmar.
\newblock Schemes over {$\bF\sb 1$}.
\newblock In {\em Number fields and function fields---two parallel worlds},
  volume 239 of {\em Progr. Math.}, pages 87--100. Birkh\"auser Boston, Boston,
  MA, 2005.

\bibitem{Deitmer:zeta-and-k-theory}
Anton Deitmar.
\newblock Remarks on zeta functions and {$K$}-theory over {${\bF}\sb 1$}.
\newblock {\em Proc. Japan Acad. Ser. A Math. Sci.}, 82(8):141--146, 2006.

\bibitem{Deligne:SGA4.5}
P.~Deligne.
\newblock {\em Cohomologie \'etale}.
\newblock Lecture Notes in Mathematics, Vol. 569. Springer-Verlag, Berlin,
  1977.
\newblock S\'eminaire de G\'eom\'etrie Alg\'ebrique du Bois-Marie SGA 4${1\over
  2}$, Avec la collaboration de J. F. Boutot, A. Grothendieck, L. Illusie et J.
  L. Verdier.

\bibitem{Deligne:HodgeI}
Pierre Deligne.
\newblock Th\'eorie de {H}odge. {I}.
\newblock In {\em Actes du Congr\`es International des Math\'ematiciens (Nice,
  1970), Tome 1}, pages 425--430. Gauthier-Villars, Paris, 1971.

\bibitem{Deligne:HodgeII}
Pierre Deligne.
\newblock Th\'eorie de {H}odge. {II}.
\newblock {\em Inst. Hautes \'Etudes Sci. Publ. Math.}, (40):5--57, 1971.

\bibitem{Deligne:HodgeIII}
Pierre Deligne.
\newblock Th\'eorie de {H}odge. {III}.
\newblock {\em Inst. Hautes \'Etudes Sci. Publ. Math.}, (44):5--77, 1974.

\bibitem{Deninger:cohomological-approach}
Christopher Deninger.
\newblock Evidence for a cohomological approach to analytic number theory.
\newblock In {\em First {E}uropean {C}ongress of {M}athematics, {V}ol.\ {I}
  ({P}aris, 1992)}, volume 119 of {\em Progr. Math.}, pages 491--510.
  Birkh\"auser, Basel, 1994.

\bibitem{Diers:book}
Yves Diers.
\newblock {\em Categories of commutative algebras}.
\newblock Oxford Science Publications. The Clarendon Press Oxford University
  Press, New York, 1992.

\bibitem{Durov:new-approach}
Nikolai Durov.
\newblock New approach to arakelov geometry.
\newblock arXiv:0704.2030, version dated 2007/04/16.

\bibitem{Fulton:Toric-book}
William Fulton.
\newblock {\em Introduction to toric varieties}, volume 131 of {\em Annals of
  Mathematics Studies}.
\newblock Princeton University Press, Princeton, NJ, 1993.

\bibitem{Greenberg:I}
Marvin~J. Greenberg.
\newblock Schemata over local rings.
\newblock {\em Ann. of Math. (2)}, 73:624--648, 1961.

\bibitem{Greenberg:II}
Marvin~J. Greenberg.
\newblock Schemata over local rings. {II}.
\newblock {\em Ann. of Math. (2)}, 78:256--266, 1963.

\bibitem{Grothendieck:Chern}
Alexander Grothendieck.
\newblock La th\'eorie des classes de {C}hern.
\newblock {\em Bull. Soc. Math. France}, 86:137--154, 1958.

\bibitem{EGA-no.11}
Alexander Grothendieck.
\newblock \'{E}l\'ements de g\'eom\'etrie alg\'ebrique. {I}{I}{I}. \'{E}tude
  cohomologique des faisceaux coh\'erents. {I}.
\newblock {\em Inst. Hautes \'Etudes Sci. Publ. Math.}, (11):167, 1961.

\bibitem{EGA-no.28}
Alexander Grothendieck.
\newblock \'{E}l\'ements de g\'eom\'etrie alg\'ebrique. {IV}. \'{E}tude locale
  des sch\'emas et des morphismes de sch\'emas. {III}.
\newblock {\em Inst. Hautes \'Etudes Sci. Publ. Math.}, (28):255, 1966.

\bibitem{Grothendieck:Pursuing-stacks}
Alexander Grothendieck.
\newblock Pursuing stacks.
\newblock Written 1983.
\newblock Unpublished.

\bibitem{Haran:mysteries-of-the-real}
M.~J.~Shai Haran.
\newblock {\em The mysteries of the real prime}, volume~25 of {\em London
  Mathematical Society Monographs. New Series}.
\newblock The Clarendon Press Oxford University Press, New York, 2001.

\bibitem{Haran:non-additive-geometry}
M.~J.~Shai Haran.
\newblock Non-additive geometry.
\newblock {\em Compos. Math.}, 143(3):618--688, 2007.

\bibitem{Illusie:dRW-1016}
Luc Illusie.
\newblock Finiteness, duality, and {K}\"unneth theorems in the cohomology of
  the de {R}ham-{W}itt complex.
\newblock In {\em Algebraic geometry (Tokyo/Kyoto, 1982)}, volume 1016 of {\em
  Lecture Notes in Math.}, pages 20--72. Springer, Berlin, 1983.

\bibitem{Kapranov:absolute-direct-image}
M.~Kapranov.
\newblock Some conjectures on the absolute direct image.
\newblock 1995 May 26. Unpublished.

\bibitem{Kapranov-Smirnov:F1}
M.~Kapranov and A.~Smirnov.
\newblock Cohomology determinants and reciprocity laws: number field case.
\newblock Unpublished.

\bibitem{Kisin-Lehrer:CountingPonts}
M.~Kisin and G.~I. Lehrer.
\newblock Equivariant {P}oincar\'e polynomials and counting points over finite
  fields.
\newblock {\em J. Algebra}, 247(2):435--451, 2002.

\bibitem{Kisin:pst}
Mark Kisin.
\newblock Potential semi-stability of {$p$}-adic \'etale cohomology.
\newblock {\em Israel J. Math.}, 129:157--173, 2002.

\bibitem{Kurokawa-Koyama:multiple-zeta}
Shin-ya Koyama and Nobushige Kurokawa.
\newblock Multiple zeta functions: the double sine function and the signed
  double {P}oisson summation formula.
\newblock {\em Compos. Math.}, 140(5):1176--1190, 2004.

\bibitem{Kurokawa:mutlple-zeta-an-example}
Nobushige Kurokawa.
\newblock Multiple zeta functions: an example.
\newblock In {\em Zeta functions in geometry ({T}okyo, 1990)}, volume~21 of
  {\em Adv. Stud. Pure Math.}, pages 219--226. Kinokuniya, Tokyo, 1992.

\bibitem{Kurokawa:zeta-functions-over-F1}
Nobushige Kurokawa.
\newblock Zeta functions over {${\Bbb F}\sb 1$}.
\newblock {\em Proc. Japan Acad. Ser. A Math. Sci.}, 81(10):180--184 (2006),
  2005.

\bibitem{Kurokawa-et-al:absolute-derivations}
Nobushige Kurokawa, Hiroyuki Ochiai, and Masato Wakayama.
\newblock Absolute derivations and zeta functions.
\newblock {\em Doc. Math.}, (Extra Vol.):565--584 (electronic), 2003.
\newblock Kazuya Kato's fiftieth birthday.

\bibitem{Laumon:Drinfeld-book-I}
G{\'e}rard Laumon.
\newblock {\em Cohomology of {D}rinfeld modular varieties. {P}art {I}},
  volume~41 of {\em Cambridge Studies in Advanced Mathematics}.
\newblock Cambridge University Press, Cambridge, 1996.
\newblock Geometry, counting of points and local harmonic analysis.

\bibitem{Manin:Cyclotomy}
Yuri Manin.
\newblock Cyclotomy and analytic geometry over ${\bF}_1$.
\newblock arXiv:0809.1564, version dated 2008/09/09.

\bibitem{Manin:Lectures-on-zeta}
Yuri Manin.
\newblock Lectures on zeta functions and motives (according to {D}eninger and
  {K}urokawa).
\newblock {\em Ast\'erisque}, (228):4, 121--163, 1995.
\newblock Columbia University Number Theory Seminar (New York, 1992).

\bibitem{Neukirch:CFT}
J{\"u}rgen Neukirch.
\newblock {\em Class field theory}, volume 280 of {\em Grundlehren der
  Mathematischen Wissenschaften [Fundamental Principles of Mathematical
  Sciences]}.
\newblock Springer-Verlag, Berlin, 1986.

\bibitem{Panjarape-Srinivas:self-maps}
K.~H. Paranjape and V.~Srinivas.
\newblock Self-maps of homogeneous spaces.
\newblock {\em Invent. Math.}, 98(2):425--444, 1989.

\bibitem{Ritt:Permutable}
J.~F. Ritt.
\newblock Permutable rational functions.
\newblock {\em Trans. Amer. Math. Soc.}, 25(3):399--448, 1923.

\bibitem{Ritt:Permutable-errata}
J.~F. Ritt.
\newblock Errata: ``{P}ermutable rational functions''.
\newblock {\em Trans. Amer. Math. Soc.}, 26(4):494, 1924.

\bibitem{Serre:CL}
Jean-Pierre Serre.
\newblock {\em Corps locaux}.
\newblock Hermann, Paris, 1968.
\newblock Deuxi\`eme \'edition, Publications de l'Universit\'e de Nancago, No.
  VIII.

\bibitem{Serre-Tate}
Jean-Pierre Serre and John Tate.
\newblock Good reduction of abelian varieties.
\newblock {\em Ann. of Math. (2)}, 88:492--517, 1968.

\bibitem{Smirnov:Hurwitz}
A.~L. Smirnov.
\newblock Hurwitz inequalities for number fields.
\newblock {\em Algebra i Analiz}, 4(2):186--209, 1992.

\bibitem{Smirnov:letter-to-Manin}
{A}lexandr~{L.} Smirnov.
\newblock Letter to {Y}. Manin. September 29, 2003.

\bibitem{Soule:F1}
Christophe Soul{\'e}.
\newblock Les vari\'et\'es sur le corps \`a un \'el\'ement.
\newblock {\em Mosc. Math. J.}, 4(1):217--244, 312, 2004.

\bibitem{Steinberg:F1}
R.~Steinberg.
\newblock A geometric approach to the representations of the full linear group
  over a {G}alois field.
\newblock {\em Trans. Amer. Math. Soc.}, 71:274--282, 1951.

\bibitem{Tate-Voloch:linear-forms}
John Tate and Jos{\'e}~Felipe Voloch.
\newblock Linear forms in {$p$}-adic roots of unity.
\newblock {\em Internat. Math. Res. Notices}, (12):589--601, 1996.

\bibitem{Tits:F1}
J.~Tits.
\newblock Sur les analogues alg\'ebriques des groupes semi-simples complexes.
\newblock In {\em Colloque d'alg\`ebre sup\'erieure, tenu \`a {B}ruxelles du 19
  au 22 d\'ecembre 1956}, Centre Belge de Recherches Math\'ematiques, pages
  261--289. \'Etablissements Ceuterick, Louvain, 1957.

\bibitem{Toen-Vaquie:Under-Spec-Z}
Bertrand To\"en and Michel Vaqui\'e.
\newblock Au-dessous de {S}pec {Z}.
\newblock arXiv:math/0509684, version dated 2007/10/05.

\bibitem{vandenBogaart-Edixhoven}
Theo van~den Bogaart and Bas Edixhoven.
\newblock Algebraic stacks whose number of points over finite fields is a
  polynomial.
\newblock In {\em Number fields and function fields---two parallel worlds},
  volume 239 of {\em Progr. Math.}, pages 39--49. Birkh\"auser Boston, Boston,
  MA, 2005.

\bibitem{Wang-CL:Cohomology-theory}
Chin-Lung Wang.
\newblock Cohomology theory in birational geometry.
\newblock {\em J. Differential Geom.}, 60(2):345--354, 2002.

\bibitem{Weil:Riemann-hyp-function-field}
Andr{\'e} Weil.
\newblock On the {R}iemann hypothesis in functionfields.
\newblock {\em Proc. Nat. Acad. Sci. U. S. A.}, 27:345--347, 1941.

\bibitem{Witt:Vectors}
Ernst Witt.
\newblock Zyklische {K}\"orper und {A}lgebren der {C}harakteristik $p$ vom
  {G}rad $p^n$. {S}truktur diskret bewerter perfekter {K}\"orper mit
  vollkommenem {R}estklassen-k\"orper der charakteristik $p$.
\newblock {\em J. Reine Angew. Math.}, (176), 1937.

\end{thebibliography}
\bibliographystyle{plain}

\end{document}